\documentclass[a4paper,twoside]{article}

\usepackage[latin1]{inputenc}
\usepackage[T1]{fontenc}
\usepackage[ngerman,english]{babel}

\usepackage{amsfonts}
\usepackage[square,numbers]{natbib}
\usepackage{amsthm}
\usepackage{amsmath}
\usepackage{amssymb}
\usepackage{hyperref}
\usepackage[all]{xy}

\usepackage{fancyhdr}
\theoremstyle{plain}
\newtheorem{thm}{Theorem}[subsection]
\newtheorem*{thmstar}{Satz}
\newtheorem*{propstar}{Proposition}
\newtheorem*{corstar}{Corollary}

\newtheorem{lem}[thm]{Lemma}
\newtheorem{prop}[thm]{Proposition}
\newtheorem{cor}[thm]{Corollary}

\theoremstyle{definition}
\newtheorem{defn}[thm]{Definition}

\theoremstyle{remark}
\newtheorem*{rem}{Remark}

\newcommand{\resettheoremcounters}{%
  \setcounter{thm}{0}%
}

\renewcommand{\tilde}{\widetilde}
\renewcommand{\rho}{\varrho}
\renewcommand{\hat}{\widehat}
\newcommand{\R}{\mathbb{R}}
\newcommand{\Rp}{\mathbb{R}^{>0}}
\newcommand{\Rpl}{\mathbb{R}^{\geq 0}}
\newcommand{\eps}{\varepsilon}
\newcommand{\Z}{\mathbb{Z}}
\newcommand{\N}{\mathbb{N}}
\newcommand{\C}{\mathbb{C}}
\newcommand{\K}{\mathbb{K}}
\newcommand{\Ham}{\mathbb{H}}
\newcommand{\Oct}{\mathbb{O}}
\newcommand{\I}{\mathrm{I}}
\newcommand{\iu}{\mathrm{i}}
\newcommand{\id}{\operatorname{id}}

\newcommand{\Det}{\operatorname{det}}
\newcommand{\End}{\operatorname{End}}
\newcommand{\tr}{\operatorname{tr}}

\newcommand{\vol}{\operatorname{vol}}
\newcommand{\den}{\omega}
\newcommand{\n}[1]{\left\lVert#1\right\rVert}
\newcommand{\norm}[1]{\left\lVert#1\right\rVert}
\newcommand{\nd}[1]{\left\lVert#1\right\rVert_{L^2_\Omega}}
\newcommand{\nl}[1]{\left\lVert#1\right\rVert_{L^2(M)}}
\newcommand{\skl}[2]{\left\langle#1,#2\right\rangle_{L^2(M)}}
\newcommand{\sklr}[2]{\left\langle#1,#2\right\rangle}

\newcommand{\mean}{\eta}
\newcommand{\volS}{\operatorname{vol}(S^{n-1})}

\newcommand{\irp}{\operatorname{injrad}(p)}
\newcommand{\ir}{\operatorname{injrad}(M)}
\newcommand{\diam}{\operatorname{diam}(M)}
\newcommand{\ric}{\operatorname{ric}}
\newcommand{\grad}{\operatorname{grad}}
\newcommand{\Cut}{\mathcal{C}}
\newcommand{\RP}{\R\emph{P}}

\newcommand{\spn}[1]{\operatorname{span}\left\{#1\right\}}
\newcommand{\conv}[1]{\operatorname{conv}\left\{#1\right\}}

\begin{document}

%
%
%
%

\title{An Introduction to Harmonic Manifolds and \\ the Lichnerowicz Conjecture}
\author{Peter Kreyssig\\
  Biosystems Analysis Group,\\
  Friedrich Schiller University of Jena,\\
  07743 Jena,\\
  Germany\\
  \texttt{peter.kreyssig@uni-jena.de}}
\date{\today}
\maketitle

\begin{abstract}

  The title is self-explanatory. We aim to give an easy to read and self-contained introduction to
  the field of harmonic manifolds. 
  Only basic knowledge of Riemannian geometry is required.
  After we gave the definition of harmonicity and derived some properties, we concentrate on 
  Z. I. Szab\'o's proof of Lichnerowicz's conjecture in the class of compact simply connected manifolds.

\end{abstract}

\section{Introduction}
\setcounter{page}{1}

\subsection{History of Lichnerowicz's Conjecture}

One attempt to find solutions of the Laplace 
equation $\Delta f=0$ is to look for them only in special classes of 
functions. It is easy to find the solutions
$$f_n:\R^n\setminus\{0\}\to\R,\;\; x\mapsto \n{x}^{2-n}$$
for $n\neq 2$ and 
$$f_2:\R^2\setminus\{0\}\to\R,\;\; x\mapsto \log\,\n{x}$$
for $n=2$ in the class of radially symmetric functions on $\R^n\setminus\{0\}$.

In 1930 H. S. Ruse gave this ansatz a try for pointed open balls in general Riemannian manifolds and thought 
he had succeeded, cf. \cite{ruse_def}.
Together with E. T. Copson he published the article \cite{copson_local} 
in which they described a mistake in Ruse's proof.
Consequently they defined, amongst other notions of harmonicity, completely harmonic space which
are nowadays called locally harmonic spaces. 
A Riemannian manifold is said to be locally harmonic if it allows a non-constant 
radially symmetric solution of the Laplace equation around every point in a small enough neighbourhood.
They also derived that this condition is equivalent to the constancy of the mean curvature 
of small geodesic spheres. Furthermore they showed that locally harmonic spaces are necessarily Einsteinian.
Hence they have constant curvature in dimensions $2$ and $3$. 
See \cite{patterson_ruse} for a detailed description of H. S. Ruse's work on 
locally harmonic manifolds.
Interestingly there are many more, fairly different, but equivalent formulations 
for harmonicity such as the validity of the mean value theorem, which was proved 
by T. J. Willmore in \cite{willmore_mean}, or the radial symmetry of the density function.

In 1944 A. Lichnerowicz conjectured that locally harmonic manifolds of dimension $4$
are necessarily locally symmetric spaces. He also gave 
some strong hints for a proof of his statement
and remarked that he did not know whether it holds in higher dimensions as well, cf. \cite[pp. 166-168]{lichnerowicz_conj}. 
In \cite[Theorem 1]{walker_four} A. G. Walker proved Lichnerowicz's original conjecture.
But since the used arguments rely heavily on the dimension, there was no hope 
to generalise them. The conjecture could be refined by A. J. Ledger since he showed that 
locally symmetric manifolds are locally harmonic if and only if they are flat or have rank $1$, 
\cite{ledger_sym}. 
So what today is called `Lichnerowicz's conjecture' was born: 
`Every locally harmonic manifold is either flat or locally symmetric of rank $1$.'
A complete collection of the knowledge about locally harmonic manifolds at its time 
was given in the book \cite{ruse_harm}.

%
%
%
%
%

An important result of global nature is due to A.-C. Allamigeon. 
He showed in \cite[p. 114]{allamigeon_blaschke}
that complete simply connected locally harmonic manifolds are either Blaschkean
or diffeomorphic to $\R^n$. This established the connection with the generalised
Blaschke conjecture, which is: `Every Blaschke manifold is 
a compact Riemannian symmetric space of rank $1$.'

Actually, there were several notions of harmonicity defined, which only coincide 
under additional topological restrictions. Amongst others we have infinitesimal, local, 
global and strong harmonicity. One uses `harmonic manifold' as a collective term 
since it is clear from the context which type of harmonicity is meant.
In \cite[Theorem 2]{michel_strong} D. Michel used Brownian motion techniques to 
show that compact simply connected globally harmonic manifolds are 
strongly harmonic. Later on Z. I. Szab\'o gave a shorter and simpler proof, cf. \cite[Theorem 1.1]{szabo_main}.
He also remarked that the notions of infinitesimal, local and global harmonicity are equivalent in the class of 
complete manifolds because of the Kazdan-DeTurck theorem, cf. \cite[Theorem 5.2]{kazdan_regularity}. 


A. L. Besse found an embedding map for strongly harmonic manifolds into a Euclidean
sphere of suitable radius, cf. \cite[Theorem 6.99]{besse_closed}. The embedded manifold 
has unexpected additional properties, e.g. it is minimal in the sphere and its geodesics are 
screw lines. Because of that it was given the name `nice embedding'. The mentioned book also
presented all of the at that time known facts about harmonic manifolds and Blaschke manifolds.

The major breakthrough was made by Z. I. Szab\'o in 1990. He proved the Lichnerowicz conjecture 
for the class of compact simply connected manifolds in his article \cite{szabo_main}.
In 2000 A. Ranjan published a slightly changed version of Z. I. Szab\'o's proof using a more careful 
analysis of a certain ODE through perturbations. The interesting aspect about this is that it makes 
no use of the nice embedding in one of the key steps of the proof, namely that the density function is a 
trigonometric polynomial of a special form, cf. \cite[Corollary 3.1]{ranjan_intrin}.
A less technical argument can be found in \cite[Theorem 2]{nikolayevsky_harm}.
Furthermore, by using a result about the first eigenvalue of P-manifolds, 
cf. \cite[Theorem 1]{ranjan_first}, one can give an intrinsic proof without using an embedding.

Surprisingly, one of the more recent results is the following. 
There are globally harmonic manifolds in infinitely many dimensions greater or equal to $7$ 
which are not locally symmetric, cf. \cite[Corollary 1]{damek_nonharm}.
E. Damek and F. Ricci constructed one-dimensional extensions of Heisenberg-type groups
which are simply connected and globally harmonic, but only symmetric if the used Heisenberg-type group has a 
centre of dimension $1$, $3$ or $7$.
This leaves the question what additional condition would be sufficient to force a 
harmonic manifold to be locally symmetric and whether there are counterexamples 
in every dimension greater or equal to $7$ .

In \cite[Theorem 1]{nikolayevsky_harm} Y. Nikolayevsky 
used the curvature conditions derived by A. J. Ledger, today called Ledger's formulae, 
cf. \cite[pp. 231-232]{willmore_geo}, to solve the conjecture in dimension $5$, i.e.~he showed that
every locally harmonic manifold of dimension $5$ has constant sectional curvature.
Namely, after lengthy and tedious calculations he is able to compute the algebraic 
curvature tensors which satisfy the first two of Ledger's formulae, yielding that they are parallel. 
Lichnerowicz's conjecture remains unsolved in dimension $6$.


A very recent result is due to J. Heber. 
In \cite[Corollary 1.2]{heber_homo} he showed that a simply connected homogeneous 
globally harmonic space is either flat, symmetric of rank $1$ or one of the non-symmetric 
spaces found by E. Damek and F. Ricci. 
This is achieved by carefully examining the structure of the group of isometries which is, 
endowed with a suitable metric, isometric to the manifold. 
First he showed that it is simply transitive and solvable and then that its commutator 
has codimension $1$. Finally his calculation of the stable Jacobi tensors yields the claim.

There are many more related topics, results and open questions not mentioned yet. Here is a
short list with some of the latest references:
harmonicity in semi-Riemannian manifolds, 
$k$-harmonicity \cite{vanhecke_k},  
infinitesimally harmonic at every point implies infinitesimally harmonic \cite{vanhecke_infini},
non-compact strongly harmonic manifolds,
commutative and D'Atri spaces \cite{berndt_damek},
Busemann functions in a harmonic manifold \cite{ranjan_busemann},
asymptotical harmonicity \cite{heber_homo}, etc.

\subsection{Extended Abstract}




This subsection contains a more detailed account of the structure of this article and its 
differences with and additions to Z. I. Szab\'o's work.

The second section gives a concise introduction to the objects and notions 
needed to examine locally harmonic manifolds. Namely, it consists of 
the definitions of Jacobi tensors along geodesics, density function, geodesic involution, mean curvature 
of geodesic spheres, radial and averaged functions and screw lines as well as 
some of their properties. The approach to screw lines presented here is due to J. von Neumann 
and I. J. Schoenberg. Hence the detour over the notion of curvatures in \cite[Section 3]{szabo_main} 
and \cite[Lemma 4.9]{szabo_main} can be avoided, cf. Lemmata \ref{congscrewlem} and \ref{lastlem}.

 
In Subsection \ref{Equivalences} we present several, rather different, but equivalent
definitions of local harmonicity, e.g. `geodesic spheres have constant mean curvature',
`every harmonic function satisfies the mean value property' and 
`the radial derivative commutes with the average operator', where the last one seems 
to be a new characterisation.
For our considerations the local version of Z. I. Szab\'o's so-called `basic commutativity'
is of greater interest. It states that 
local harmonicity is equivalent to the commutating of the average operator 
with the Laplace operator. Its global version is used to find radial eigenfunctions 
of the Laplacian later on.
We also prove that locally harmonic manifolds are Einsteinian. Hence they are analytic 
by the Kazdan-DeTurck theorem. Then we can show that the density function does not depend on the point.

Section \ref{Blaschke} contains some basic facts about Blaschke 
manifolds and a proof of the (original) global version of the basic commutativity. 
We use a different argument to Z. I. Szab\'o's one, cf. \cite[p. 5]{szabo_main}, 
since we only show that the radialised average is $C^2$ and not $C^\infty$, cf. Lemma \ref{basiccom2}.

The next aim is to understand the relation between the notions of
locally, globally and strongly harmonic manifolds. Important for our 
argumentation is that they coincide under the hypothesis of a 
compact simply connected manifold and that we then get the 
Blaschke property. 
 
Then we show that averaged eigenfunctions are solutions of a certain linear ODE 
involving the mean curvature by using the basic commutativity. This yields 
some findings on the structure of the spaces of (radial) eigenfunctions.
Also contained in Section \ref{Eigenfunctions} is a characterisation of 
local harmonicity in Blaschke manifolds by means of the $L^2$-product.

In Section \ref{embedding} we show that locally harmonic Blaschke manifolds which are not diffeomorphic to a sphere
can be embedded into a Euclidean sphere of suitable radius, cf. Corollary \ref{bessenicecor}.
This is Z. I. Szab\'o's new version of Besse's so-called `nice embedding' using a radial eigenfunction.
In \cite[Theorem 3.1]{szabo_main} it is stated with a weaker hypothesis, but without mentioning
the exception of the sphere.


Finally, we are ready to prove the main result.
\begin{thmstar}[main result]
  Let $M$ be locally harmonic Blaschke manifold of dimension $n$ and diameter $\pi$.
  Then $M$, and therefore every compact simply connected locally harmonic manifold,
  is a Riemannian symmetric space of rank $1$, i.e.~isometric (up to scaling of the metric) to either 
  $S^n$, $\C\text{\emph P}^{\frac{n}{2}}$, $\Ham\text{\emph P}^{\frac{n}{4}}$ or $\Oct\text{\emph P}^2$. 
\end{thmstar}
Z. I. Szab\'o showed that the averaged eigenfunctions of the Laplacian 
can be written as polynomials in cosine by showing that the space spanned by their 
parallel displacements is finite-dimensional.
The same is true for the square of the density function. Here he used the 
embedding theorem to be able to carry out calculations in a Euclidean space, cf. Lemma \cite[Lemma 4.3]{szabo_main}.
We present a slightly varied version of Y. Nikolayevsky's proof of this statement which does not 
make use of an embedding, cf. Lemma \ref{denpolylem}.
Then Z. I. Szab\'o derived restrictions to the possible roots of the mentioned polynomials. 
This rather technical part uses essentially the aforementioned linear ODE solved by the 
averaged eigenfunctions.
Note that we give a new proof for \cite[Lemma 4.6]{szabo_main}, cf. Lemma \ref{rootdistrilem}.
Consequently there is a strong restriction to the form of the density function and hence to the form 
of the mean curvature. 
\begin{propstar} 
  There are $\alpha,\beta\geq 0$ such that
  $$ \mean_p(q)=\frac{(\alpha+ \beta)\cos d(p,q)+\beta}{\sin d(p,q)},$$
  where $\mean_p(q)$ is the mean curvature of the geodesic sphere 
  of radius $0<d(p,q)<\pi$ around $p\in M$ in the point $q\in M$.
\end{propstar}
Again by using the ODE this enabled Z. I. Szab\'o to find the spectrum and the radial eigenfunctions easily.
This is the content of Subsection \ref{specnradeigen}.
\begin{corstar}
  The spectrum $(\lambda_k)_{k\in\N_0}$ of $M$ is given by 
  $\lambda_k:=k(k+\alpha+\beta)$. A radial eigenfunction 
  to $\lambda_1$ around $p\in M$ is given by
  $$M\ni q \mapsto \frac{\lambda_1}{n}\left(\cos d(p,q)+\frac{n-\lambda_1}{\lambda_1}\right).$$ 
\end{corstar}
In Subsection \ref{var1} we show that this radial eigenfunction to the first eigenvalue 
yields an especially nice embedding. More precisely,
the geodesics are mapped into circles so that the geodesic symmetries are isometries. Hence the 
main result is established.
Alternatively, one can use the Bott-Samelson theorem \cite[Theorem 7.23]{besse_closed} and the statement of 
\cite[Theorem 1]{ranjan_first} to give an intrinsic version of the proof, cf. Subsection \ref{var2}.

\newpage

\subsection{Notations and Conventions}

In this subsection we fix some notations and general hypotheses, which 
are valid for the whole article. This is meant to serve the reader as
a guideline and to give them a feeling for the used notations. \\

Let $(M,g)$ be a connected Riemannian manifold of dimension $n$ 
with metric $g$. The Levi-Civita connection will be denoted by $\nabla$.
Denote by $T_pM$ the tangent space in $p\in M$ and by 
$TM$ the tangent bundle of $M$. Points in $TM$ will be denoted
by $(p,v)$ where $p\in M$ and $v\in T_pM$.\\


The geodesic distance between two points $p,q\in M$ will be denoted 
by $d(p,q)$.
The metric sphere of radius $R\ge 0$ around $p\in M$ is then given by 
$S^d_R(p):=\{q\in M \;|\; d(p,q)=R\}$. \\

We denote the cut locus of $p\in M$ by $\Cut(p)$. We write $\irp$ for the 
injectivity radius of $M$ at $p$ and $\ir$ for the injectivity radius
of $M$. The diameter of $M$ is denoted by $\diam$. \\

We also use the standard notation for the function spaces 
$L^2(M)$, $C^0(M)$, $C^\infty(M)$, $C^0([0,\infty[)$, 
$C^\infty([0,\infty[)$, $\dots$ and the space $\ell^2$ of
square-integrable sequences.\\

For an eigenvalue $\lambda\in\R$ of the Laplacian 
$\Delta$ we have the space of 
eigenfunctions $V^\lambda\subset C^\infty(M)$. \\

We abbreviate `Riemannian symmetric space of rank $1$' by `ROSS'. 
These are the Euclidean spheres $S^n$, the projective spaces $\K\text{P}^m$ and $\Oct\text{P}^2$ and 
the hyperbolic spaces $\K\text{H}^m$ and $\Oct\text{H}^2$, where $\K\in\{\R,\C,\Ham\}$. Here $m$ denotes 
the $\K$-dimension of the respective space, i.e.~$m\cdot\dim_{\R}(\K)=n$.\\

We use $\volS$ for the volume of the sphere of radius $1$ in $\R^n$.\\

The open geodesic ball of radius $0<R\le\irp$ around $p\in M$ is denoted by 
$B_R(p)$. The related ball in $T_pM$ is denoted by $B_R(0_p)$.
Furthermore, set $\hat B_R(p):=B_R(p)\setminus \{p\}$ and 
$\hat B_R(0_p):=B_R(0_p)\setminus \{0_p\}$
for the pointed balls.
Similarly the geodesic sphere $S_R(p)$ in $M$ and the related
sphere in $S_R(0_p)$ in $T_pM$ of radius $0<R<\irp$ are defined.\\

Polar coordinates are used throughout this article, i.e.~for a $v\in \hat B_R(0_p)$ where $0<R\le \irp$ we often
write $v=r\theta$ where $r:=\n{v}$ and $\theta:=\frac{v}{r}\in S_1(0_p)$. \\





%

For a smooth curve $\gamma$ in $M$ we denote by $T^\perp\gamma$ the subbundle 
of $\gamma^*TM$ normal to $\gamma'$. Furthermore we define a section $R_\gamma$ of 
$\End(T^\perp\gamma)$ by $R_\gamma = R(\cdot,\gamma')\gamma'$ where $R$ is the 
curvature tensor. For a section $S\in \Gamma(\End(T^\perp\gamma))$ of the endomorphism bundle
we set $S':=\nabla_{\gamma'}S$ where $\nabla$ is used for the induced connection on $\End(T^\perp\gamma)$.

\section{Preliminaries}

This section contains a big chunk of 
the necessary setup, as the definitions and some properties 
of the relevant objects in conjunction with
locally harmonic manifolds are given.
The most important results are the invariance of the density function
under the geodesic involution (Lemma \ref{deninvlem}), an equation which relates 
mean curvature and the density (Lemma \ref{meandenlem}) and a formula for the Laplacian 
of radial functions (Lemma \ref{lapofradlem}). 
In the last subsection we show that two screw lines are congruent if and only if they have 
got the same screw function (Lemma \ref{congscrewlem}).


\subsection{Jacobi Tensors}

The concept of Jacobi tensors comes in handy later on because it 
reduces complexity of notation. A useful reference is \cite[Section 2]{eschenburg_jacobi}.
Let $\gamma$ be a geodesic in $M$ and assume that $0$ is in its domain of definition.

\begin{defn}[Jacobi tensor]
  We call a section $J$ of the endomorphism bundle $\End(T^\perp\gamma)$ which satisfies
  $$J''+R_\gamma\circ J=0$$
  a \emph{Jacobi tensor to $\gamma$}.
\end{defn}

\begin{rem}
  Set $p:=\gamma(0)$. Take a basis $(e_2,\dots,e_n)$ of $T^\perp_p\gamma$ and denote by $(E_2,\dots,E_n)$ 
  its parallel translate along $\gamma$. Choose Jacobi fields $J_2,\dots,J_n$
  along $\gamma$ with $J_i(0),J_i'(0)\in T^\perp_p\gamma$ where $i=2,\dots,n$. We can define a
  Jacobi tensor $J$ to $\gamma$ by setting $JE_i:=J_i$ for $i=2,\dots,n$.
  It is easy to see that every Jacobi tensor to $\gamma$ can be written that way.
  If $(e_2,\dots,e_n)$ is an orthonormal basis of $T^\perp_p\gamma$, we get 
  $$JE_i=\sum_{j=2}^n g(J_i,E_j)E_j$$
  and if $(e_2,\dots,e_n)$ is additionally positively oriented 
  $$\Det J=\Det \left(g(J_i,E_j)\right)_{i,j=2,\dots,n}.$$
\end{rem}

\begin{defn}[associated Jacobi tensor]
  There is exactly one Jacobi tensor $J$ to $\gamma$ with $J(0)=0$ and $J'(0)=\id$.
  We call it the \emph{Jacobi tensor associated to $\gamma$}.
\end{defn}

\subsection{Density Function}


Local harmonicity is defined in terms of the density function, which will be 
examined in this subsection. From its definition it is not immediately clear why 
the density function is smooth and why it is called `density'. Therefore we 
give a formula for it in normal coordinates, which clarifies the situation.
The results of this subsection can also be found in \cite[Section 6.6]{willmore_geo}. 

\begin{defn}[density function]
  Choose $V\subset TM$ such that $\exp:V\to M$ is defined.
  Let $(p,v)\in V$ with $v\neq 0$ and set $\tilde v:=\frac{v}{\n{v}}$. 
  Let $J_v$ be the Jacobi tensor associated to the normalised geodesic $r\mapsto \exp_pr\tilde v$.
  The \emph{density function $\den$} is then defined by
  $$\den: V \to \R, \;\; (p,v)\mapsto \norm{v}^{1-n}\Det_{T^\perp_{\exp_p v}\gamma} \left(J_v(\norm{v})\right)$$
  where we set $\den(p,0_p):=1$.
\end{defn}

\begin{rem}
  The density function $\den$ is obviously continuous on $V$ and $\den(p,v)=0$ if
   and only if $p$ and $\exp_pv$ are conjugate along $r\mapsto \exp_pr\tilde v$.
  When fixing a point $p\in M$ and choosing a normal coordinate neighbourhood 
  $U$ around $p$, we will often write $\den_p(q):=\den(p,\exp_p^{-1}q)$ 
  for $q\in U$.
  If $M$ is complete, $\den$ is defined on the whole of $TM$.
\end{rem}

\begin{lem}[density in normal coordinates] \label{denincoords}
  Let $U$ be a normal neighbourhood around $p\in M$.
  Take $q\in U$ and let $(g_{q,ij})_{i,j=1\dots n}$ be the metric of $T_qM$ 
  expressed in the normal coordinates of $U$. Then 
  $$\den_p(q) = \sqrt{\det (g_{q,ij})_{i,j=1\dots n}}.$$
\end{lem}



\begin{proof}

  We have $\den_p(p)=1=\sqrt{\det (g_{p,ij})_{i,j=1\dots n}}$.
  So assume $q\neq p$ and set $v:=\exp_p^{-1} q$ as well as $e_1:= \frac{v}{\n{v}}$. 
  Pick $e_2,\dots,e_n\in T_pM$ such that $(e_1,\dots,e_n)$ is a positively oriented orthonormal basis of $T_pM$.
  We identify this basis with the standard basis in $\R^n$.
  Denote by $J_1,\dots,J_n$ the Jacobi fields 
  along the geodesic $r\mapsto \exp_p re_1$ with initial conditions $J_i(0)=0_p$ and $J_i'(0)=e_i$ where 
  $i=1,\dots,n$. We get
   $$g_{q,ij} = g_{\exp_pv,ij} = g_{\exp_pv}\left( (d\exp_p)_v(e_i),(d\exp_p)_v(e_j) \right) 
   = \frac{1}{\n{v}^2} g_{\exp_pv}\left( J_i(\n{v}),J_j(\n{v})\right).$$
  Taking the determinant yields 
  $$\det (g_{q,ij})_{i,j=1\dots n}=\den_p(q)^2.$$
  The claim follows since $\den_p$ is positive on $U$.

%
%
%
%
%
%
%
%
%
\end{proof}

\begin{rem}
  This lemma shows that $\den$ is smooth in inner 
  points of its domain. Additionally it explains why we call $\den$ the density function since
  the Riemannian volume is defined by integration of $\den_p$.
\end{rem}

\subsection{Geodesic Involution}

In this subsection we show the invariance of the density under
the geodesic involution. This result is important for the 
proof of Proposition \ref{denpointprop}. It is also contained 
in \cite[Section 6.B]{besse_closed}.

\begin{defn}[(canonical) geodesic involution]
  Let $V\subset TM$ be the maximal subset of the tangent bundle such that $\exp:V\to M$
  is defined. The \emph{(canonical) geodesic involution} $i$ is then 
  defined by
  $$i:V\to V,\;\; (p,v)\mapsto \left(\exp_p(v), -(d\exp_p)_v(v)\right).$$
\end{defn}

\begin{rem}
  Indeed, this is well-defined as $i(V)\subset V$ and an involution as $i(i(p,v))=(p,v)$.
\end{rem}

\begin{lem}[density invariant under geodesic involution] \label{deninvlem}
  Let $V\subset TM$ be the maximal subset of the tangent bundle such that $\exp:V\to M$
  is defined. Then
  $$\forall\; (p,v)\in V:\;\; \den(p,v)=\den(i(p,v)).$$
\end{lem}

\begin{proof}
  For $v=0_p$ the statement is true because $i(p,0_p)=(p,0_p)$. 

  So consider $(p,v)\in V$ with $v\neq 0_p$.
  Set $\tilde v:=\frac{v}{\n{v}}$ and for $r\in[0,\n{v}]$ set
  $\gamma(r):=\exp_p\left(r\tilde v\right)$.
  The density function in the point $(p,v)$ can be written as 
  $$\den(p,v) = \n{v}^{1-n}\Det_{T^\perp_{\gamma(\n{v})}\gamma}\left(J(\n{v})\right)$$
  where $J\in \Gamma(\End(T^\perp\gamma))$ is the Jacobi tensor associated to $\gamma$.
  By setting $\overline\gamma(r):=\exp_{\exp_p(v)}(-r(d\exp_p)_v(\tilde v))$ for $r\in[0,\n{v}]$ we get
  $$\den(i(p,v)) = \norm{v}^{1-n}\Det_{T^\perp_p\gamma}(\overline K(\n{v}))$$
  where $\overline K\in \Gamma(\End(T^\perp\overline\gamma))$ is the Jacobi tensor associated to $\overline\gamma$.
  Define the section $K$ of $\End(T^\perp\gamma)$ by
  $K(r):=\overline K(\n{v}-r)$ for $r\in[0,\n{v}]$.
  We remark that 
  $K'' + R_{\gamma}\circ K= 0$
  holds
  because of 
  $\nabla_{\gamma'}=\nabla_{-\overline\gamma'}=-\nabla_{\overline \gamma'}$
  and
  $R_{\overline\gamma}(\n{v}-r)=R_\gamma(r)$ for $r\in[0,\n{v}]$.

  Then
  $$\mathcal{J}:=(J^T)'\circ K - J^T\circ K'$$
  is a section of $\End(T^\perp\gamma)$ where $(\cdot)^T$ means 
  transposition of an endomorphism. We have
  \begin{align*}
    \mathcal{J}' &= ((J^T)'\circ K - J^T\circ  K')' \\
    &= (J^T)''\circ  K + (J^T)'\circ  K'
    - (J^T)'\circ  K'- J^T\circ  K'' \\
    &= (J^T)''\circ  K - J^T\circ  K'' \\
    &= -(R_\gamma\circ J)^T\circ  K 
    + J^T\circ (R_{\gamma} \circ  K) \\
    &= -J^T\circ R_\gamma^T\circ  K 
    + J^T\circ R_\gamma\circ  K \\
    &= -J^T\circ R_\gamma\circ  K 
    + J^T\circ R_\gamma \circ  K \\
    &= 0.
  \end{align*}
  Hence the section $\mathcal{J}$ is parallel along $\gamma$. 

  Because of 
  $$\mathcal{J}(0)=((J^T)'\circ  K)(0) - (J^T\circ  K')(0)= K(0)$$
  and 
  $$\mathcal{J}(\n{v})=((J^T)'\circ  K)(\n{v}) - (J^T\circ  K')(\n{v})=(J^T)'(\n{v})\circ \overline K(0) + J^T(\n{v})\circ \overline K'(0)=J^T(\n{v})$$ 
  we get that $J^T(\n{v})$ is the parallel translate of $K(0)$ along $\gamma$.
  That means 
  \begin{align*}
    \den(p,v) &= \norm{v}^{1-n}\Det_{T^\perp_{\gamma(\n{v})}\gamma}(J(\n{v})) \\
    &= \norm{v}^{1-n}\Det_{T^\perp_{\gamma(\n{v})}\gamma}\left(J^T(\n{v})\right) \\
    &= \norm{v}^{1-n}\Det_{T^\perp_p\gamma}(K(0)) \\
    &= \norm{v}^{1-n}\Det_{T^\perp_p\gamma}(\overline K(\n{v})) \\
    &= \den(i(p,v)).
  \end{align*}
\end{proof}

%
%
%
%
%
%


\subsection{Mean Curvature}

This subsection describes the relation between the mean 
curvature of geodesic spheres and the density function. Lemma \ref{meandenlem} is central
for the proof of various equivalences in the next section and the proof 
of Lichnerowicz's conjecture. A useful reference is \cite[Section 2]{eschenburg_jacobi}.

%
%


\begin{defn}[mean curvature (of geodesic spheres)] \label{meandef}
  Let $q\in \hat B_R(p)$ be a point in the pointed geodesic ball of radius $0<R\le \irp$ around $p\in M$.
  Set $v:=\exp_p^{-1}q$ and $\tilde v:= \frac{v}{\n{v}}$. 
  Let $J_v$ be the Jacobi tensor associated to the geodesic $r\mapsto \exp_pr\tilde v$.
  The \emph{mean curvature $\mean_p(q)$ (of the geodesic sphere $S_{\n{v}}(p)$) in the point $q$} is 
  defined by
  $$\mean_p(q):=\tr(J_v'\circ J_v^{-1})(\n{v}).$$
  
\end{defn}

\begin{rem}
  Define the section $\mathcal{S}_v\in\Gamma(\End(T^\perp\gamma))$ by
  $\Gamma(T^\perp\gamma)\ni X\mapsto \nabla_X\gamma'\in \Gamma(T^\perp\gamma).$
  Then $\mathcal{S}_v(\n{v})$ is the shape operator of $S_{\n{v}}(p)$ in the point $q$.
  Because of
  $$\Gamma(T^\perp\gamma)\ni J_v' X-\mathcal{S}_v J_v X=\nabla_{\gamma'}J_v X -\nabla_{J_v X} \gamma'=\left[\gamma',J_v X\right]\perp \Gamma(T^\perp\gamma)$$
  we get $J_v'=\mathcal{S}_v\circ J_v$. Hence our definition of $\mean_p$ coincides with the one usually given
  as the trace of the shape operator.
  We have 
  $$\sum_{i=2}^n \nabla^\perp_{E_i}E_i =\sum_{i=2}^n g(\gamma',\nabla_{E_i}E_i)\gamma' =-\sum_{i=2}^n g(\nabla_{E_i}\gamma',E_i)\gamma'  = -\mean_p \gamma'$$
  where $E_2,\dots,E_n$ are fields along $\gamma$ such that $(\gamma',E_2,\dots,E_n)$ is 
  orthonormal along $\gamma$ and $\nabla^\perp$ denotes the part of the connection tangent to $\gamma'$, i.e.~normal to the geodesic spheres.
\end{rem}


\begin{lem}[mean curvature through density] \label{meandenlem}
  For $q\in \hat B_R(p)$ as above set again $v:=\exp_p^{-1}q$, $r:=\n{v}$ and $\tilde v:=\frac{v}{r}$. Then
  $$\mean_p(q)=\frac{\partial_r \left(r^{n-1}\den\left(p,r\tilde v\right)\right)}{r^{n-1}\den(p,r\tilde v)}=\frac{n-1}{r}+\frac{\partial_r\den(p,r\tilde v)}{\den(p,r\tilde v)}.$$
\end{lem}

\begin{proof}
  The first equality follows from the formula
  $$(\Det J_v)'=\tr (J_v'\circ J_v^{-1}) \Det(J_v).$$
  Hence
  $$\mean_p(q) = \frac{\partial_r \left(r^{n-1}\den\left(p,r\tilde v\right)\right)}{r^{n-1}\den(p,r\tilde v)} 
  = \frac{(n-1)r^{n-2}\den(p,r\tilde v) + r^{n-1}\partial_r\den(p,r\tilde v)}{r^{n-1}\den(p,r\tilde v)}
  =\frac{n-1}{r}+\frac{\partial_r\den(p,r\tilde v)}{\den(p,r\tilde v)}.$$
\end{proof}


\subsection{Radial and Averaged Functions}

Note that we only consider
functions on pointed geodesic balls in this subsection. More general considerations
are given for the special case of a Blaschke manifold later on.
Strictly speaking, there are no results in this subsection except of
Lemma \ref{lapofradlem}. We only define some notions for the following discussion.
Fix a point $p\in M$ and a number $0<R\le \irp$. 

\begin{defn}[normal and outward vector field]
  Denote by $E^p$ the \emph{normal and outward vector field} of 
  $\hat B_R(p)$ which is given by $(E^p)_q:=(d\exp_p)_v\left(\frac{v}{\n{v}}\right)$
  for $q\in\hat B_R(p)$ with $v:=\exp_p^{-1}q$.
\end{defn}


\begin{rem}
  $E^p$ is the unique unit vector field on $\hat B_R(p)$ such that $E^p$ is
  normal and outward along $S_r(p)$ for all $0 < r < R$.
\end{rem}

\begin{defn}[(associated) radial function]
  For a smooth function $F:\left]0,R\right[\to \R$ we define
  the \emph{(associated) radial function (around $p\in M$) on $\hat B_R(p)$} by 
  $$R_pF:\hat B_R(p)\to \R,\;\; q\mapsto F(d(p,q)).$$
  We call $R_p:C^\infty(]0,R[)\to C^\infty(\hat B_R(p))$
  \emph{radial operator (around $p$)}.
  Functions $f:\hat B_R(p)\to \R$ such that an $F:\left]0,R\right[\to\R$ exists
  with $f=R_pF$ are called \emph{radially symmetric functions (around $p$)} 
  or abbreviated \emph{radial functions (around $p$)}.
\end{defn}

\begin{rem}
  The radial operator is linear.
\end{rem}

\begin{defn}[average operator]
  Let $f:\hat B_R(p)\to \R$ be smooth. The \emph{averaged function $A_pf$ of $f$ (around $p\in M$)} is defined by
  $$A_pf:\left]0,R\right[\to \R,\;\;r\mapsto (A_pf)(r):=\frac{1}{\text{vol}(S_r(p))}\int_{S_r(p)}f|_{S_r(p)}\;dS_r(p).$$
  We call $A_p:C^\infty(\hat B_R(p))\to C^\infty(]0,R[)$ \emph{average operator (around $p$)}.
\end{defn}

\begin{rem}
  The average operator is linear.

\end{rem}

\begin{defn}[radial derivative]
  Let $E^p$ be the normal and outward vector field of $\hat B_R(p)$. We define 
  the \emph{radial derivative $f'$ of $f:\hat B_R(p)\to \R$} by
  $$f':\hat B_R(p)\to \R,\;\;q\mapsto f'(q):=(\nabla_{E^p}f)(q).$$
\end{defn}

\begin{rem}
  In terms of polar coordinates and the exponential map we can 
  write $f'(\exp_pr\theta)=\partial_rf(\exp_pr\theta)$
  where $0<r<R$ and $\theta\in S_1(0_p)$.
\end{rem}

\begin{lem}[properties of the radial operator]
  Let $h:\hat B_R(p)\to \R$ and $F,G:\left]0,R\right[\to \R$. Then
  \begin{enumerate}
    \item $A_pR_p F = F$
    \item $R_p (FG) = R_pF R_pG$
    \item $A_p(hR_pG)=GA_ph$
    \item $(R_p F)' = R_pF'$
  \end{enumerate}
\end{lem}

\begin{proof}
  The first three statements are clear. 

  Using the above remark we have 
  for $q\in \hat B_R(p)$ with $q=\exp_pr\theta$
  $$(R_pF)'(q) = \partial_r (R_pF)(\exp_pr\theta) = \partial_r F(r) = (R_pF')(q).$$
\end{proof}

\begin{lem}[Laplacian of radial functions] \label{lapofradlem}
  Let $f:\hat B_R(p)\to\R$ be a radial function. Then 
  $$\Delta f = -f'' - \mean_pf'.$$
\end{lem}

\begin{proof}
  Fix $0<r<R$ and let $q\in S_r(p)$. Denote the connection on $S_r(p)$ by 
  $\overline \nabla$ and the associated Laplacian by $\overline \Delta$. Since 
  $f$ is radial, $f|_{S_r(p)}$ is constant and $\overline \Delta f|_{S_r(p)}=0$.
  Take $e_2,\dots,e_n\in T_qM$ such that $((E^p)_q,e_2,\dots,e_n)$ is an orthonormal basis of $T_qM$.
  In the point $q$ we get
  \begin{align*}
    (\Delta f)(q) &= -\nabla^2_{(E^p)_q,(E^p)_q} f-\sum_{i=2}^n \nabla^2_{e_i,e_i} f \\
    &= -f''(q)-\sum_{i=2}^n \left(\nabla_{e_i}\nabla_{e_i} f - (\nabla_{e_i}e_i)f \right) \\
    &= -f''(q)-\sum_{i=2}^n \left(\overline\nabla^2_{e_i,e_i} f - (\nabla^\perp_{e_i}e_i)f\right) \\
    &= -f''(q)+(\overline\Delta f)(q) + \sum_{i=2}^n (\nabla^\perp_{e_i}e_i)f \\
    &= -f''(q) - \mean_p(q)\nabla_{(E^p)_q}f \\
    &= -f''(q) - \mean_p(q)f'(q).
  \end{align*}
\end{proof}

\begin{rem}
  In particular it holds $\Delta d(p,\cdot) = -\mean_p$ on $\hat B_{\irp}(p)$.
\end{rem}

\subsection{Screw Lines}

Let $N\in\N$ and $c:\R\to \R^N$ be a smooth curve which is 
parametrised by arc length. 
We will discuss some kind of generalisation of curves with
constant curvatures called screw lines. This is needed when discussing 
the nice embedding.
The following Lemma \ref{congscrewlem} is true for curves $c:\R\to\ell^2$ as well.
The ideas can also be found in \cite[Part II]{neumann_screw}.

\begin{defn}[screw function and screw line]
  We define the 
  \emph{screw function $S_{s_0}$ in $s_0\in \R$ of $c$} by 
  $$S_{s_0}:\R\to \R,\;\; s\mapsto \n{c(s_0+s)-c(s_0)}^2.$$ 
  The curve $c$ is called \emph{screw line} if its screw functions are
  independent of the chosen points, i.e.
  $$\forall\; s_0\in\R :\;\; S_{s_0}=S_0.$$ 
\end{defn}

\begin{lem} \label{congscrewlem}
  Let $c$ and $\overline c$ be screw lines
  which have the 
  same screw function. Then they are congruent, i.e.~there is an 
  isometry $I\in \text{\emph{Iso}}(\R^N)$ with 
  $I(c(s))=\overline c(s)$ for all $s\in\R$.
\end{lem}

\begin{proof}
  Firstly, we remark that for all $r,s,t\in\R$ holds
  $$\sklr{c(t)-c(r)}{c(s)-c(r)} = \frac{1}{2} (S_0(t-r) + S_0(s-r)-S_0(t-r-(s-r)))=\sklr{\overline c(t)-\overline c(r)}{\overline c(s)-\overline c(r)}.$$

  Without loss of generality we may assume that $c(0)=0=\overline c(0)$.
  We choose $t_1,\dots,t_k\in\R$ such that $(c(t_1),\dots,c(t_k))$ is a basis of the space 
  $\spn{c(t)\;|\;t\in\R}$. By applying the Gram-Schmidt process to this basis we get an orthonormal basis 
  $(e_1,\dots,e_k)$. We denote by $a_{ij}\in \R$ the coefficients of the change of basis given 
  by that process, i.e.~
  $$e_i=\sum_{j=1}^i a_{ij}c(t_j),\;\;i=1,\dots,k.$$ 
  We emphasise that the $a_{ij}$'s only depend on the scalar products
  $$\sklr{c(t_\nu)}{c(t_\mu)},\;\;\nu,\mu=1,\dots,k.$$
  Furthermore we fix an $s\in\R$, write 
  $$c(s)=\sum_{i=1}^k b_i(s)e_i$$
  and emphasise that the $b_i(s)$'s only depend on the scalar products
  $$\sklr{c(t_\nu)}{c(t)},\;\;\nu=1,\dots,k,\;\; t\in\R.$$

  Because of our first remark we get that $(\overline c(t_1),\dots,\overline c(t_k))$ is a basis of the space 
  $\spn{\overline c(t)\;|\;t\in\R}$ and $(\overline e_1,\dots,\overline e_k)$ with 
  $$\overline e_i:=\sum_{j=1}^i a_{ij}\overline c(t_j),\;\;i=1,\dots,k$$ 
  is the orthonormal basis we get by applying the Gram-Schmidt process. Furthermore it holds
  $$\overline c(s)=\sum_{i=1}^k b_i(s)\overline e_i.$$

  Let $A\in O(N)$ be an orthonormal transformation mapping
  $e_i$ into $\overline e_i$ for $i=1,\dots,k$.
  We get
  $$Ac(s)=\sum_{i=1}^k b_i(s)Ae_i = \sum_{i=1}^k b_i(s)\overline e_i =\overline c(s).$$
\end{proof}

\section{Local Harmonicity}

A rough definition for $M$ being locally harmonic could be
`locally the density function is radially symmetric'. The aim of this section 
is to state the definition more precisely and to give several characterisations of locally
harmonic manifolds. Especially Parts (2.)~and (6.)~of Proposition \ref{equivalences} are
important for our considerations. Furthermore we give examples and show that locally harmonic
manifolds are Einsteinian (Proposition \ref{harmeinsteinprop}).


\subsection{Definition and Equivalences} \label{Equivalences}

We give several equivalent definitions of a locally harmonic manifold.
Note that we show with Corollary \ref{analyticcor} that the following proposition is still true if we 
formulate it with $\irp$ instead of $\eps$. 
The basic commutativity (Proposition \ref{equivalences}(6.)) can be found in \cite[Section 1]{szabo_main}.
The commuting of the averaging operator with the radial derivative 
(Proposition \ref{equivalences}(3.)) seems to be nowhere mentioned. The rest of Proposition \ref{equivalences} can be
found in \cite[Proposition 6.21]{besse_closed}. 

\begin{defn}[locally harmonic]
  The Riemannian manifold $M$ is said to be \emph{locally harmonic at $p\in M$} 
  if there exists an $\eps>0$ such that $\den_p|_{\hat B_\eps(p)}$ is radial.
  If $M$ is locally harmonic at every point, we call it \emph{locally harmonic}.
\end{defn}

\begin{rem}
  Equivalently, we could require the existence of an $\Omega: [0,\eps[\to \R$ such that
  $$\forall\; v\in B_\eps(0_p):\;\; \den(p,v) = \Omega(\n{v}).$$
  Notice that the choice of $\eps$ and $\Omega$ could depend on $p$. 
  Actually, it does not, as we will prove in Proposition \ref{denpointprop}.
  The property `locally harmonic' is often abbreviated by `LH'. A manifold which is 
  LH is often called LH-manifold.
\end{rem}


\begin{prop}[equivalences]  \label{equivalences}
  Let $p\in M$. Then the following statements are equivalent:
  \begin{enumerate}
    \item $M$ is locally harmonic at $p$. 
    \item There is an $\eps>0$ and an $H:\left]0,\eps\right[\to \R$ with
      $\mean_p=R_pH$, i.e.~the mean curvature is radial.
    \item There is an $\eps>0$ such that for every $f:\hat B_{\eps}(p)\to \R$ 
      we have $(A_pf)'=A_pf'$, i.e.~the radial derivative commutes 
      with the average operator.
    \item There is an $\eps>0$ such that for every $f\in C^\infty(\hat B_\eps(p))$ with 
      $\Delta f=0$ we have $(A_pf)'=0$, i.e.~every harmonic function satisfies the mean value property.
    \item There is an $\eps>0$ and a non-constant $F:\left]0,\eps\right[\to \R$ with
      $\Delta R_pF=0$, i.e.~there is a non-constant radial solution of the Laplace equation.
    \item There is an $\eps>0$ such that for every $f:\hat B_{\eps}(p)\to \R$ 
      we have $\Delta R_pA_pf=R_pA_p \Delta f$, i.e.~the Laplace operator
      commutes with $R_p\circ A_p$.
  \end{enumerate}
\end{prop}

\begin{proof}
  $\emph{1.} \Rightarrow \emph{2.:}$
  Choose an $\eps$ such that $\den_p:\hat B_\eps(p)\to\R$ is radial. 
  Then so is $\den_p':\hat B_\eps(p)\to\R$. By Lemma \ref{meandenlem} the 
  mean curvature is radial, too. 

  $\emph{2.} \Rightarrow \emph{1.:}$
  Choose $\eps>0$ and $H:\left]0,\eps\right[\to\R$ such that $\mean_p=R_pH$.
  Let $\hat H$ be the antiderivative of $H-\frac{n-1}{\id}$ in $]0,\eps[$. 
  Let $\theta\in S_1(p)$.
  The solution
  of the ODE
  $$\frac{y'}{y}=H-\frac{n-1}{\id}$$
  with initial condition 
  $$y\left(\frac{\eps}{2}\right)=\den\left(p,\frac{\eps}{2}\theta\right)$$
  is given by $y(r)=C(\theta)\exp(\hat H(r))$ for $r\in \left]0,\eps\right[$ where $C(\theta)$ is a 
  constant depending on $\theta$.
  Since $r\mapsto \den(p,r\theta)$ solves the ODE as well, we have 
  $\den(p,r\theta)=C(\theta)\exp(\hat H(r))$ for $r\in \left]0,\eps\right[$.
  Because $\den$ is continuous in $(p,0_p)$ with $\den(p,0_p)=1$ we get that 
  $$\lim_{r\to 0} C(\theta)\exp(\hat H(r))$$
  exists and equals $1$. Hence $C(\theta)$ does not depend on $\theta$ and $\den_p|_{\hat B_\eps(p)}$ is radial.

  $\emph{2.} \Rightarrow \emph{3.:}$
  Choose $\eps>0$, $H:\left]0,\eps\right[\to\R$ and $\Omega:\left]0,\eps\right[\to\R$ 
  such that $\mean_p=R_pH$ and $\den_p=R_p\Omega$ on $\hat B_\eps(p)$.
  Let $0< r < \eps$. 
  By taking polar coordinates and Lemma \ref{denincoords} into account we have 
  \begin{align*}
    (A_pf)(r) &= \frac{1}{\int_{S_1(0_p)} r^{n-1}\den(p,r\theta)\;d\theta}\int_{S_1(0_p)} f(\exp_p(r\theta))r^{n-1}\den(p,r\theta)\;d\theta \\
    &= \frac{1}{\int_{S_1(0_p)} r^{n-1}\Omega(r)\;d\theta}\int_{S_1(0_p)} f(\exp_p(r\theta))r^{n-1}\Omega(r)\;d\theta \\
    &= \frac{1}{\volS}\int_{S_1(0_p)} f(\exp_p(r\theta))\;d\theta.
  \end{align*}
  Taking the derivative yields the claim.

  $\emph{3.} \Rightarrow \emph{4.:}$
  Choose an $\eps>0$ such that for every $f\in C^\infty(\hat B_\eps(p))$ we have $(A_pf)'=A_pf'$. Suppose $\Delta f=0$.
  Hence for every $0<r<\eps$ we get by Green's first identity
  \begin{align*}
    (A_pf)'(r) = (A_pf')(r) &= \frac{1}{\vol(S_r(p))}\int_{S_r(p)} \nabla_{E^p} f \;dS_r(p) \\
     &= \frac{1}{\vol(S_r(p))}\int_{S_r(p)} \sklr{\grad f}{E^p} \;dS_r(p) \\
     &= -\frac{1}{\vol(S_r(p))}\int_{\hat B_r(p)} \Delta f \;d\hat B_r(p) \\
     &= 0.
  \end{align*}

  $\emph{4.} \Rightarrow \emph{2.:}$
  Choose an $\eps>0$ such that for every $f\in C^\infty(\hat B_\eps(p))$ with $\Delta f=0$ we have $(A_pf)'=0$.
  We set 
  $$H(r):=\frac{\partial_r\vol(S_r(p))}{\vol(S_r(p))}$$
  and show that $\mean_p =R_pH$. Take an $0<r<\eps$. By solving a Dirichlet problem 
  we can find an $f\in C^\infty(\overline{B_r(p)})$ with
  $\Delta f|_{\hat B_r(p)}=0$ and $f|_{S_r(p)}=\mean_p-R_pH$. Because of 
  \begin{align*}
   0 = \vol(S_r(p))(A_pf)'(r) &= - \frac{\partial_r \vol(S_r(p))}{\vol(S_r(p))} \int_{S_r(p)} f \;dS_r(p)
  + \partial_r \int_{S_r(p)} f \;dS_r(p)\\
   &= -\int_{S_r(p)} f R_pH \;dS_r(p)  +  \partial_r \int_{S_1(0_p)} f(\exp_p(r\theta))r^{n-1}\den(p,r\theta)\;d\theta\\
   &= -\int_{S_r(p)} f R_pH \;dS_r(p)  +  \int_{S_r(p)} f' \;dS_r(p) + \int_{S_r(p)} f\mean_p \;dS_r(p)\\
   &= -\int_{S_r(p)} f R_pH \;dS_r(p)  -  \int_{B_r(p)} \Delta f \;dB_r(p) + \int_{S_r(p)} f\mean_p \;dS_r(p)\displaybreak\\
   &= \int_{S_r(p)} (\mean_p - R_pH)^2 \;dS_r(p) 
  \end{align*}
  the claim follows. 

  $\emph{2.} \Rightarrow \emph{5.:}$
  Choose an $\eps>0$ such that $\mean_p:\hat B_\eps(p)\to\R$ is radial and 
  a function $H:\left]0,\eps\right[\to\R$ with $R_pH=\mean_p$.
  Let $F:\left]0,\eps\right[\to \R$ be a non-constant solution of the ODE 
  $$-y''-Hy'=0.$$
  We have
  $$\Delta R_pF = -(R_pF)''-\mean_p(R_pF)'=-R_pF''-R_pHR_pF'=R_p(-F''-HF')=0.$$

  $\emph{5.} \Rightarrow \emph{2.:}$
  Take an $\eps>0$ and a non-constant $F:\left]0,\eps\right[\to \R$ with $\Delta R_pF=0$.
  Since 
  $$0=\Delta R_pF = -R_pF''-\mean_pR_pF'$$
  we have 
  $$\mean_pR_pF'=-R_pF''$$
  and 
  $$0=-A_pR_pF''-A_p\mean_pA_pR_pF'=-F''-A_p\mean_pF'.$$
  If $F'$ had a zero $0<r_0<\eps$, $F$ would be constant, since $F$ would be a solution of the ODE
  $-y''-A_p\mean_py'=0$
  with $F'(r_0)=F''(r_0)=0$.
  So $\mean_p$ is radial with
  $$\mean_p=R_p\left(\frac{-F''}{F'}\right).$$

  $\emph{2.} \Rightarrow \emph{6.:}$
  Choose $\eps>0$ and $H:\left]0,\eps\right[\to\R$ such that $\mean_p=R_pH$.
  For a fixed $0<r<\eps$ denote the Laplacian on $S_r(p)$ by $\overline \Delta$. 
  As in the proof of Lemma \ref{lapofradlem} we get for a $q\in S_r(p)$
  $$(\Delta f)(q) = (\overline\Delta f)(q) - f''(q) -\mean_p(q)f'(q).$$
  By Green's first identity we have
  $$\int_{S_r(p)} \overline\Delta f|_{S_r(p)}\;dS_r(p) = 0$$
  and therefore again in $q\in S_r(p)$
  \begin{align*} 
    (R_pA_p \Delta f)(q) &= R_pA_p( (\overline\Delta f)(q) - f''(q) -\mean_p(q)f'(q)) \\
    &= -(R_pA_pf'')(q)-\mean_p(q)(R_pA_pf')(q) \\
    &= -(R_p(A_pf)'')(q)-\mean_p(q)(R_p(A_pf)')(q)\\
    &= (\Delta R_pA_p f)(q).
  \end{align*}

  $\emph{6.} \Rightarrow \emph{2.:}$
  Choose an $\eps>0$ such that for every $f:\hat B_{\eps}(p)\to \R$ 
  we have $\Delta R_pA_pf=R_pA_p \Delta f$.
  If we set $f:=d(p,\cdot)$, we get
  $$R_pA_p \Delta d(p,\cdot) = \Delta R_pA_p d(p,\cdot) = \Delta d(p,\cdot)= -\mean_p.$$
  This means that the mean curvature is a radial function.

\end{proof}




\subsection{Curvature Restrictions}

The main result of this subsection is that 
LH-manifolds are Einsteinian and therefore analytic.
The proof for this statement can be 
found in \cite[Section 6.8]{willmore_geo}. Furthermore, we can deduce that 
in an LH-manifold the density function $\den(p,v)$ does not depend on the point $p$, cf. \cite[Proposition 6.7.3]{willmore_geo}.
In this section we let $V\subset TM$ be the maximal subset of the tangent bundle such that $\exp:V\to M$ is defined.

\begin{prop}[harmonic manifolds are Einsteinian] \label{harmeinsteinprop}
  Every LH-manifold is an Einstein manifold.
\end{prop}

\begin{proof}

  Fix $p\in M$ and $\theta\in S_1(0_p)$.
  Choose an $\eps>0$ such that $\mean_p:\hat B_\eps(p)\to\R$ is radial and 
  a function $H:\left]0,\eps\right[\to\R$ with $R_pH=\mean_p$.
  For $r\in[0,\eps]$ set $\gamma(r):=\exp_p r\theta$.
  Denote by $J$ the Jacobi tensor associated to $\gamma$. 
  The inverse tensor $J^{-1}$ has got a singularity of order $n-1$ in $0$
  because of $\lim_{r\to 0} r^{1-n}\det J(r) = \den(p,0_p)=1$.
  So the section $\mathcal{J}:=rJ'\circ J^{-1}$ of $\End(T^\perp\gamma)$ is not singular in $0$.
  We get 
  $$r\mathcal{J}'=rJ'\circ J^{-1}+r^2J''\circ J^{-1}-r^2J'\circ J^{-1}\circ J'\circ J^{-1}=\mathcal{J}-r^2R_\gamma-\mathcal{J}^2$$
  since $(J^{-1})' = -J^{-1}\circ J'\circ J^{-1}$.
  Differentiating the equation $r\mathcal{J}'=\mathcal{J}-r^2R_\gamma-\mathcal{J}^2$ yields 
  $$\mathcal{J}'+ r\mathcal{J}''=\mathcal{J}'-2rR_\gamma-r^2R'_\gamma-\mathcal{J}'\circ\mathcal{J}-\mathcal{J}\circ\mathcal{J}'$$
  and differentiating once more yields
  $$\mathcal{J}''+ r\mathcal{J}'''=-2R_\gamma-2rR'_\gamma-2rR'_\gamma-r^2R''_\gamma-\mathcal{J}''\circ\mathcal{J}-\mathcal{J}'\circ\mathcal{J}'-\mathcal{J}'\circ\mathcal{J}'-\mathcal{J}\circ\mathcal{J}''.$$

  Since $\lim_{r\to 0} \frac{J(r)}{r} = J'(0)$ we get from the definition of $\mathcal{J}$ and the two equations above
  $$\mathcal{J}(0)=\id,\qquad \mathcal{J}'(0)=0\qquad \text{and}\qquad \mathcal{J}''(0)=-\frac{2}{3}R_\gamma(0).$$
  Taking the trace in the last equation gives
  $$-\frac{2}{3}\ric_p(\theta,\theta)=\tr \mathcal{J}''(0)=(\tr \mathcal{J})''(0)=(rH(r))''(0).$$
  This shows that $\ric_p(\theta,\theta)$ does not depend on the chosen 
  $\theta$. Hence $M$ is Einsteinian.

\end{proof}


\begin{rem}
  In dimensions $2$ and $3$ this implies that $M$ has constant sectional curvature. 
  Taking more and more derivatives of $r\mathcal{J}'=\mathcal{J}-r^2R_\gamma-\mathcal{J}^2$ yields the so-called
  `Ledger's formulae', cf. \cite[Section 6.8]{willmore_geo}. With their help one can give an 
  affirmative answer to Lichnerowicz's conjecture in dimension
  $4$, cf. \cite[Section 6.E]{besse_closed}. 
\end{rem}

\begin{thm}[{Kazdan-DeTurck, \cite[Theorem 5.2]{kazdan_regularity}}]
  Let $(M,g)$ be an Einstein manifold. Then the representation of $g$ in 
  normal coordinates is real analytic. 
\end{thm}

\begin{rem}
  This implies that normal coordinates define a real analytic atlas on $M$. 
  So we see that the map $\exp:\operatorname{int}V \to M$ is real analytic by using normal coordinates.
\end{rem}

\begin{cor}[density function is analytic] \label{analyticcor}
  Let $(M,g)$ be an LH-manifold.
  Then the density function $\den:\operatorname{int}V \to \R$ is real analytic.
\end{cor}

\begin{proof}
  The density is given by a composition of the real analytic functions 
  $d\exp$, $\det$ and $g$.
\end{proof}

\begin{rem}
  We emphasise that only now we know that the density $\den_p$ of an LH-manifold 
  is radial till the injectivity radius and that $\den(p,v)$ only depends on $\n{v}$ for $(p,v)\in V$.
\end{rem}

\begin{prop}[density independent of the point] \label{denpointprop}
  Let $M$ be an LH-manifold.
  Then there is a function $\Omega:[0,\infty[\to \R$ 
  such that
  $$\forall (p,v)\in V: \;\; \den(p,v)=\Omega(\n{v}).$$
\end{prop}

\begin{proof}
  Let $\sigma:[0,1]\to M$ be a smooth curve in $M$. Set 
  $$\delta:=\frac{1}{2}\min_{t\in[0,1]}\operatorname{injrad}(\sigma(t))$$
  and 
  $$U:=\bigcup_{t\in[0,1]}B_\delta(\sigma(t)).$$
  Then $U$ is open and connected. The density $\den(p,r\theta)$ is defined for
  $p\in U$, $\theta\in S_1(0_p)$ and $0\leq r < \delta$.
  Pick an $0\leq r < \delta$ and define $\overline \den(r,\cdot):U\to \R$ by 
  $\overline \den(r,p):=\den(p,r\theta)$. This is well-defined, i.e. does not 
  depend on $\theta\in S_1(0_p)$, because of the local harmonicity of $M$. 
  
  We will show that for every $p\in U$ with $\overline \den(r,p)\neq 0$ the derivative, 
  namely $(d\overline \den(r,\cdot))_p:T_pM\to \R$, vanishes. This implies that 
  $\overline \den(r,\cdot)$
  is constant on the components of 
  $U\setminus\overline \den(r,\cdot)^{-1}(\{0\})$. By the connectedness of $U$ and the continuity of 
  $\overline \den(r,\cdot)$
  we get the following. 
  In the case $\overline \den(r,\cdot)^{-1}(\{0\})=\emptyset$ we have a constant
  $\overline \den(r,\cdot)$.
  In the case $\overline \den(r,\cdot)^{-1}(\{0\})\neq \emptyset$ we have
  $\overline \den(r,\cdot)=0$.

  Let $u\in T_pM$. In order to show
  $(d\overline \den(r,\cdot))_p(u)=0$
  we construct a curve through $p$ with initial velocity $u$.
  Take a normalised geodesic $\gamma:[0,r]\to M$ with 
  $\gamma(r)=p$ and $g_p(\gamma'(r),u)=0$. Set $q:=\gamma(0)$.
  Because of $\den(p,-r\gamma'(r))=\overline \den(r,p)\neq 0$ the points $p$ and $q$ 
  are not conjugate along $\gamma$. Choose an $\eps>0$ and a one-parameter family of geodesics 
  $\gamma_s$ with $s\in \left]-\eps,\eps\right[$ such that
  $\gamma_s(0)= q$ for $s\in \left]-\eps,\eps\right[$ and 
  $$\left.\frac{d}{ds}\right|_{s=0}\gamma_s(r)=u.$$
  By the invariance under the geodesic involution (Lemma \ref{deninvlem}) we have 
  $$\overline \den(r,q)=\den(q,r\gamma_s'(0))=\den(\gamma_s(r),-r\gamma_s'(r))=\overline \den(r,\gamma_s(r)).$$
  Hence
  $$(d\overline \den(r,\cdot))_p(u)=\left.\frac{d}{ds}\right|_{s=0}\overline \den(r,\gamma_s(r))=\left.\frac{d}{ds}\right|_{s=0}\overline \den(r,q)=0.$$
  We get that 
  $\overline \den(r,\cdot)$ is constant on $U$ and therefore
  $$\den(\sigma(0),r\theta)=\overline \den(r,\sigma(0))=\overline \den(r,\sigma(1))=\den(\sigma(1),r\theta)$$ 
  for $0\leq r < \delta$.
  By the above Lemma \ref{analyticcor} we get the claim.


\end{proof}



\subsection{Examples}

We compute the density functions of the ROSSs,
cf. \cite[Section 3.E]{besse_closed}, and show that 
locally symmetric spaces of rank $1$ are examples of LH-manifolds.


%
%
%

\begin{prop}[density functions of the ROSSs] \label{denofrossprop}
  Let $p\in M$ and $\theta\in S_1(0_p)$.
  Set $d(\K)=\dim_\R(\K)$ for $\K\in\{\R,\C,\Ham\}$ and denote by $m$ the $\K$-dimension of the ROSSs.
  If we assume that the hyperbolic spaces have sectional curvature between $-1$ and $-\frac{1}{4}$
  we get
  \begin{center}
  \begin{tabular}{r|c|c|c}
     $M$ & $\R\text{\emph{H}}^m$ & $\K\text{\emph{H}}^m$ & $\Oct\text{\emph{H}}^2$ \\\hline
     $r^{d(\K)m-1}\den(p,r\theta)$ & $(\sinh r)^{m-1}$ & $(\sinh r)^{d(\K)-1} (2\sinh\frac{r}{2})^{d(\K)(m-1)}$ & $(\sinh r)^7 (2\sinh\frac{r}{2})^8$
  \end{tabular}
  \end{center}
  for $0\le r<\infty$
  and if we assume that the projective spaces have diameter $\pi$ we get
  \begin{center}
  \begin{tabular}{r|c|c|c}
     $M$ & $S^m$ & $\K\text{\emph{P}}^m$ & $\Oct\text{\emph{P}}^2$ \\\hline
     $r^{d(\K)m-1}\den(p,r\theta)$ & $(\sin r)^{m-1}$ & $2^{\frac{d(\K)}{2}(m-1)}(\sin r)^{d(\K)-1} (1-\cos r)^{\frac{d(\K)}{2}(m-1)}$ & $16(\sin r)^7 (1-\cos r)^4$
  \end{tabular}
  \end{center}
  for $0\le r\le \pi$.
\end{prop}

\begin{proof}
  We only consider $M:=\C\text{P}^m$ since the computations for the other spaces work similarly.
  Choose a geodesic $\gamma$ with $\gamma(0)=p$ and $\gamma'(0)=\theta$. We denote 
  the imaginary unit by $\iu$. Choose 
  $e_3,\dots,e_{2m}\in T_p^\perp\gamma$ such that $(\theta\cdot \iu,e_3,\dots,e_{2m})$ is a basis of $T_p^\perp\gamma$ 
  in which $R_\gamma(0)$ is diagonal. Denote by $(E_2,E_3,\dots,E_{2m})$ the parallel translate of 
  $(\theta\cdot \iu,e_3,\dots,e_{2m})$ along $\gamma$. Then $R_\gamma$ is diagonal in the basis $(E_2,E_3,\dots,E_{2m})$
  since $R$ is parallel.

  In order to compute the Jacobi fields along $\gamma$ we need the eigenvalues of $R_\gamma(0)$. 
  They are $1$ and $\frac{1}{4}$ because of 
  $$g_p(R(\theta\cdot \iu,\theta)\theta,\theta\cdot \iu)=1$$ 
  and 
  $$g_p(R(e_j,\theta)\theta,e_j)=\frac{1}{4}, \;\; j=3,\dots,2m.$$ 
  So $J_2(r):=(\sin r) E_2(r)$ and $J_j(r):=(2\sin\frac{r}{2}) E_j(r)$ are Jacobi fields along $\gamma$ with the
  initial conditions $J_2(0)=0$, $J_2'(0)=\theta\cdot \iu$ and $J_j(0)=0$, $J_j'(0)=e_j$ where $j=3,\dots,2m$.
  Hence 
  $$r^{2m-1}\den(p,r\theta)=(\sin r) \left(2\sin\frac{r}{2}\right)^{2(m-1)}=2^{m-1}(\sin r)(1-\cos r)^{m-1}.$$
\end{proof}

\begin{cor}[locally symmetric spaces and local harmonicity]
  Let $M$ be a locally Riemannian symmetric space. Then $M$ is LH if and only if it is of rank $1$ or flat.
\end{cor}

\begin{proof}
  If $M$ is LH and not of rank $1$, it is flat, cf. \cite{eschenburg_sym} or \cite{ledger_sym}.
  Since for every point in a locally symmetric space there is a neighbourhood which is isometric to a neighbourhood in a 
  symmetric space, we are done by the above lemma.
\end{proof}


\section{Blaschke Manifolds} \label{Blaschke}

The aim of this section is to provide the definition and some properties
of Blaschke manifolds, since we will show that 
compact simply connected LH-manifolds are of that type in the next section. 
Noteworthy are 
Propositions \ref{blaschkeinjrad} and \ref{blaschkesc} and the 
(global) basic commutativity (Theorem \ref{basiccom2}).

\subsection{Definition and Some Properties}

We do not present any proofs in this subsection and 
refer to \cite[Sections 5.D and 5.E]{besse_closed} for a detailed account.

\begin{defn}[spherical cut locus]
  We say that $M$ has \emph{spherical cut locus at $p\in M$} if 
  $S^d_{\irp}(p) = \Cut(p)$.
\end{defn}

\begin{defn}[Blaschke manifold]
  We say that $M$ is a \emph{Blaschke manifold} if $M$ is compact and has 
  spherical cut locus at every $p\in M$. 
\end{defn}

\begin{prop}[metric spheres are submanifolds]
  In a Blaschke manifold every metric sphere is a submanifold.
\end{prop}

\begin{prop}[$\ir = \diam$] \label{blaschkeinjrad}
  For a Blaschke manifold we have $\ir=\diam=d(p,q)=\irp$ 
  where $p\in M$ and $q\in \Cut(p)$.
\end{prop}

\begin{prop}[simple and closed geodesics] \label{blaschkesc}
  In a Blaschke manifold every geodesic is simple and closed with length $2\diam$.
\end{prop}

\begin{prop}[special case: singleton cut locus] \label{blaschkespecial}
  Let $M$ be a Blaschke manifold and $p\in M$. Assume that the cut locus consists 
  of only one element, i.e.~$\Cut(p)=\{q_p\}$. 
  Then the following statements hold.
  \begin{enumerate}
    \item $M$ is diffeomorphic to the sphere $S^n$.
    \item The map 
      $$\sigma:M\to M, \;\; p\mapsto q_p$$
      is an involutive isometry.
    \item The Riemannian quotient $\overline M:=M/\sigma$ is Blaschkean and diffeomorphic to $\RP^n$.
    \item The natural projection map $\pi:M\to \overline M$
      is the universal Riemannian covering of $\overline M$.
  \end{enumerate}
\end{prop}

\begin{rem}
  Actually, $M$ is isometric to the sphere in this case, cf. \cite[Theorem D.1]{besse_closed}.
\end{rem}

\subsection{Radial and Averaged Functions}

Let $M$ be a Blaschke manifold and set $D:=\diam$.
The pieces of notation we define in this subsection are used in the following 
argumentation in the context of a Blaschke manifold only.
Note that the definitions given here coincide with the ones given earlier
on pointed open geodesic balls. Anyway, the results provided here are mostly only
true for Blaschke manifolds, cf. \cite[Section 1]{szabo_main}.

\begin{defn}[(associated) radial function]
  For a smooth function $F:[0,D]\to \R$ we define
  the \emph{(associated) radial function (around $p\in M$)} by 
  $$R_pF:M\to \R,\; q\mapsto F(d(p,q)).$$
  We call $R_p:C^\infty([0,D])\to C^0(M)\cap C^\infty(\hat B_D(p))$
  \emph{radial operator}.
  Functions $f\in C^\infty(M)$ such that an $F:[0,D]\to\R$ exists
  with $f=R_pF$ are called \emph{radially symmetric functions (around $p$)} 
  or abbreviated \emph{radial functions (around $p$)}.
\end{defn}

\begin{rem}
  The radial operator is linear. We emphasise that the function $R_pF$ 
  is not necessarily differentiable in $p$ nor in points of $\Cut(p)$.
\end{rem}

\begin{lem}[criterium for smoothness] \label{smoothcrit}
  Let $F:[0,D]\to\R$ be a smooth function. For every $p\in M$ the following two statements are
  equivalent.
  \begin{enumerate}
    \item $R_pF$ is of class $C^{2m}$.
    \item $F^{(2i-1)}(0)=F^{(2i-1)}(D)=0$ holds for $i=1,\dots,m$.
  \end{enumerate}
\end{lem}

\begin{proof}
  $\emph{1.} \Rightarrow \emph{2.:}$
  This is clear.

  $\emph{2.} \Rightarrow \emph{1.:}$
  Set $k:=\dim \Cut(p)$ and write $\R^n=\R^{n-k}\times \R^k$.
  The function $R_pF$ is certainly of class $C^{2m}$ in $B_D(p)$.
  So pick a point $q\in \Cut(p)$. Since $\Cut(p)$ is a submanifold and 
  geodesics emanating from $p$ hit the cut locus $\Cut(p)$ 
  orthogonally, we can find a chart $(\varphi,U)$ around $q$
  such that 
  \begin{enumerate}
    \item $\varphi(q)=0$,
    \item $\varphi:U\cap \Cut(p) \longrightarrow \{0\}\times \R^k\cap \varphi(U)$ is a diffeomorphism,
    \item $\varphi:U\setminus \Cut(p) \longrightarrow \R^{n-k}\times \{0\}\cap \varphi(U)$ is a diffeomorphism and
    \item For every geodesic $\gamma:\R\to M$ through $p$ and $\overline q\in \Cut(p)$ the set
      $\varphi(\gamma(\R)\cap U)$ is a line through $\varphi(\overline q)$ which is orthogonal
      to $\{0\}\times \R^k$.
  \end{enumerate}
   The function $R_pF \circ \varphi^{-1}$ is therefore of class $C^{2m}$ since its 
   partial derivatives of order $2m$ exist and are continuous.
\end{proof}

\begin{defn}[average operator]
  Let $f:M\to \R$ be a smooth function. The \emph{averaged function $A_pf$ of $f$ (around $p\in M$)} 
  is defined by
  $$A_pf:[0,D]\to \R,\;\;r\mapsto (A_pf)(r):=\lim_{\rho\to r} \left(A_p\left(f|_{\hat B_D(p)}\right)\right)(\rho).$$
  We call $A_p:C^\infty(M)\to C^\infty([0,D])$ \emph{average operator (around $p$)}.
\end{defn}

\begin{rem}
  The average operator is linear and we have $(A_pf)(0)=f(p)$.
 If we lift 
 $f|_{\Cut(p)}$ to a function $\tilde f := f\circ \exp_p$ on $S_D(0_p)$, we easily see that the average of 
 $f$ taken over the cut locus of $p$ equals the average of $\tilde f$ taken over $S_D(0_p)$. 
 So the limit equals the actual average, i.e.
 $$(A_pf)(D)= \frac{1}{\vol(\Cut(p))}\int_{\Cut(p)} f|_{\Cut(p)} \;d\Cut(p).$$
\end{rem}

\begin{lem}[properties of the radial operator]
  Let $h:M\to \R$ and $F,G:[0,D]\to \R$ be smooth and $p\in M$.
  \begin{enumerate}
    \item $A_pR_p F = F$
    \item $R_p (FG) = R_pF R_pG$
    \item $A_p(hR_pG)=GA_ph$
  \end{enumerate}
\end{lem}


%
%
%


%

\begin{lem}[(global) basic commutativity] \label{basiccom2}
  Let $M$ be a locally harmonic Blaschke manifold and $p\in M$. For every smooth function $f$
  on $M$ the function $R_pA_pf:M\to\R$ is of class $C^2$ and it holds
  $$\Delta R_pA_pf = R_pA_p\Delta f.$$
\end{lem}

\begin{proof}
  Since the equality holds on $\hat B_D(p)$, we only need to prove the first claim.

  By Lemma \ref{smoothcrit} we only need to show that $(A_pf)'(0)=(A_pf)'(D)=0$. 
  Let $\Omega:\left]0,D\right[\to \R$ be the function with $R_p\Omega = \den_p$. For $0<r<D$ we have by Green's first identity
  \begin{align*}
    (A_pf)'(r) &= -\frac{1}{\vol(S_r(p))}\int_{B_r(p)} \Delta f \;dB_r(p) \\
    &= -\frac{1}{\int_{S_1(0_p)} r^{n-1}\Omega(r) \;d\theta}\int_0^r\int_{S_1(0_p)} (\Delta f)(\exp_p\rho\theta)\rho^{n-1}\Omega(\rho) \;d\theta d\rho.
  \end{align*}
  Hence
  $$|(A_pf)'(r)| \le r\left|\frac{\max_{0\le\rho\le r} \left(\rho^{n-1} \Omega(\rho) \max_{\theta\in S_1(0_p)} (\Delta f)(\exp_p\rho\theta)\right)}{r^{n-1}\Omega(r)}\right| \le r\left|\max_{q\in B_r(p)} (\Delta f)(q)\right|$$
  and
  $$\lim_{r\to 0}(A_pf)'(r) = 0.$$
  Because of 
  $$0=\int_M \Delta f \;dM=\int_0^D\int_{S_1(0_p)} (\Delta f)(\exp_p\rho\theta)\rho^{n-1}\Omega(\rho) \;d\theta d\rho$$ 
  we get for $0<r<D$
  \begin{align*}
    |(A_pf)'(r)| &= \left|\frac{1}{\int_{S_1(0_p)} r^{n-1}\Omega(r) \;d\theta} \int_M \Delta f \;dM + (A_pf)'(r)\right| \\
    &= \left|\frac{1}{\int_{S_1(0_p)} r^{n-1}\Omega(r) \;d\theta} \int_r^D\int_{S_1(0_p)} (\Delta f)(\exp_p\rho\theta)\rho^{n-1}\Omega(\rho) \;d\theta d\rho \right| \\
    &\le (D-r)\left|\max_{q\in M\setminus B_r(p)} (\Delta f)(q)\right|.
  \end{align*}
  This proves the claim because
  $$\lim_{r\to D}(A_pf)'(r) = 0.$$
\end{proof}

\begin{rem}
  $R_pA_pf$ is actually smooth, but this fact is not needed below.
\end{rem}

\section{Other Notions of Harmonicity}

There are two more kinds of harmonicity which are of interest for 
our considerations. In this section we give the definitions for 
globally harmonic and strongly harmonic manifolds as well as 
topological conditions which force LH-manifolds to be globally 
respectively strongly harmonic. Noteworthy are Allamigeon's theorem 
(Theorem \ref{Allamigeon}) and Proposition \ref{LHimplSH}.

\subsection{Globally Harmonic Manifolds}

The most important result of global nature for LH-manifolds is Allamigeon's 
theorem, cf. \cite[Theorem 6.82]{besse_closed}, which allows us to use the statements of the previous section.

\begin{defn}[globally harmonic]
  A complete Riemannian manifold $M$ is said to be \emph{globally harmonic} if for 
  every $p\in M$ there exists $\Omega: \Rpl\to \R$ such that
  $$\forall\; v\in T_pM:\; \den(p,v) = \Omega(\n{v})$$
\end{defn}

\begin{rem}
  Notice that the choice of $\Omega$ could depend on $p$. Actually, it
  does not by Proposition \ref{denpointprop}.
  The property `globally harmonic' is often abbreviated by `GH'. A manifold which is
  GH is often called GH-manifold.
\end{rem}

\begin{prop}[LH-manifolds are GH] \label{LHimplGH}
  Every complete LH-manifold $M$ is GH.
\end{prop}

\begin{proof}
  Let $p\in M$. The density function $\den(p,\cdot)$ is an analytic function 
  $T_pM\to \R$. Since it is radially symmetric in a neighbourhood around $0_p$, 
  it is radially symmetric on the whole of $T_pM$.
\end{proof}

\begin{thm}[Allamigeon's theorem] \label{Allamigeon}
  Every complete simply connected LH-mani\-fold $M$
  is either a Blaschke manifold or diffeomorphic to $\R^n$.
\end{thm}

\begin{proof}
  By the previous lemma we know that $M$ is GH. Let $p\in M$. For every 
  $0\neq v\in T_pM$ set $\gamma_v(r):=\exp_p\left(r\frac{v}{\n{v}}\right)$ 
  for $r\in \Rpl$. Suppose there is no conjugate point along $\gamma_v$ for all 
  $0\neq v\in T_pM$. Then $\exp_p:T_pM\to M$ is a covering map and, since 
  $M$ is simply connected, a diffeomorphism. 
  
  So take a $0\neq v_0\in T_pM$ and an $r_0\in\Rpl$ such that the first conjugate point along 
  $\gamma_{v_0}$ is $\gamma_{v_0}(r_0)$. Then the first 
  conjugate point along $\gamma_v$ is $\gamma_{v}(r_0)$ for all $0\neq v\in T_pM$, since 
  $\den(p,\cdot)$ is radial. Note that $r_0$ is the same for every point in $M$.
  This means that $M$ is a Blaschke manifold by the Allamigeon-Warner theorem, cf. \cite[Corollary 5.31]{besse_closed}.
\end{proof}

\subsection{Strongly Harmonic Manifolds}

The interesting result of this subsection is Proposition 
\ref{LHimplSH}, which can also be found in \cite[Theorem 1.1]{szabo_main}.
However, we do not need any of the following statements for our discussion.

\begin{thm}[heat kernel]
  Let $M$ be a compact Riemannian manifold. 
  There exists a unique $k:\Rp\times M \times M\to \R$ with
  the following properties.
  \begin{enumerate}
    \item $k$ is continuous, of class $C^1$ in the first variable and of class $C^2$ in the second.
    \item $$\forall\; t\in\Rp\;\;\forall\; q\in M:\;\;(\partial_t + \Delta) k(t,\cdot,q) = 0.$$
    \item $$\forall\; f\in C^\infty(M)\;\;\forall\; q\in M:\;\;\lim_{t\to 0} \int_M k(t,\cdot,q)f\;dM = f(q).$$
  \end{enumerate}
  This $k$ is actually smooth and $k(t,p,q)=k(t,q,p)$ holds for every $t\in\Rp$ and $p,q\in M$.
\end{thm}

\begin{rem}
  A proof can be found in \cite[Section III.E]{berger_spectre}.
\end{rem}



\begin{defn}[strongly harmonic]
  A compact Riemannian manifold $M$ is said to be \emph{strongly harmonic} if 
  for every $t\in \Rp$ there exists a $K_t: \Rpl\to \R$ such that
  $$\forall\; t\in \Rp\;\;\forall\; p,q\in M:\; k(t,p,q) = K_t(d(p,q)).$$
\end{defn}

\begin{rem}
  The property `strongly harmonic' is often abbreviated by `SH'. A manifold which is 
  SH is often called SH-manifold.

  Since a unique heat kernel also exists in the non-compact case, we 
  could define a notion of strong harmonicity in this case as well, cf. 
  \cite[Theorem 3.5]{strichartz_complete} and \cite[p. 7]{szabo_main}, but 
  this is not needed in the following considerations.
\end{rem}

%

\begin{prop}[SH-manifolds are GH] \label{SHimplGH}
  Every strongly harmonic manifold is globally harmonic.
\end{prop}

\begin{proof}
  It suffices to show that $M$ is locally harmonic. 
  For every $t\in\Rp$ define $K_t:\left]0,\ir\right[\to \R$ such that 
  $k(t,\cdot,q)=R_qK_t$ for every $q\in M$.
  Then we have
  $$R_q\partial_tK_t = \partial_t R_qK_t = \partial_t k(t,\cdot,q)= -\Delta k(t,\cdot,q) = -\Delta R_qK_t = R_qK_t'' + \mean_q R_qK_t'.$$
  In particular $K_t:\left]0,\ir\right[\to \R$ is a solution of a linear ODE of second order. 
  Furthermore $K_t'$ is non-zero in a dense subset of $]0,\ir[$ since otherwise $K_t$
  would be constant and $\partial_tK_t$ would be zero, which would contradict the third 
  property of the heat kernel.
  Hence $\mean_q$ is radial.

%

%
%
%

\end{proof}

\begin{thm}[LH-manifolds are SH] \label{LHimplSH}
  Every compact simply connected LH-manifold is strongly harmonic.
\end{thm}

\begin{proof}
  We know that $M$ is globally harmonic and a Blaschke manifold of diameter say $D$.
  It suffices to show that $\overline k:\Rp\times M\times M\to \R,\;\;(t,p,q)\mapsto 
  \overline k(t,p,q):=(R_qA_qk(t,\cdot,q))(p)$ 
  also satisfies the properties of the heat kernel $k$, since it is
  unique. 
  Pick $t\in\Rp$ and $p,q\in M$.
  The function $\overline k$ is continuous, of class $C^1$ in the first variable 
  and of class $C^2$ in the second.
  We have 
  \begin{align*}
    \partial_t\overline k(t,\cdot,q) &= \partial_tR_qA_qk(t,\cdot,q)= R_qA_q\partial_tk(t,\cdot,q) = -R_qA_q\Delta k(t,\cdot,q)
    = -\Delta R_qA_qk(t,\cdot,q)\\
    &= -\Delta \overline k(t,\cdot,q)
  \end{align*}
  and 
  \begin{align*}
    \lim_{t\to 0} \int_M \overline k(t,\cdot,q)f\;dM &= \lim_{t\to 0} \int_M R_qA_qk(t,\cdot,q)f\;dM \\
    &= \lim_{t\to 0} \int_0^D(A_qk(t,\cdot,q))(r) \int_{S_1(0_q)} f(\exp_q r\theta) r^{n-1}\den(q,r\theta) \;d\theta dr \\
    &= \lim_{t\to 0} \int_0^D(A_qk(t,\cdot,q))(r) \vol(S_r(q)) (A_qf)(\exp_q r\theta) \;dr \\
    &= \lim_{t\to 0} \int_0^D(A_qf)(r) \int_{S_1(0_q)} k(t,\exp_q r\theta,q)r^{n-1}\den(q,r\theta) \;d\theta dr \\
    &= \lim_{t\to 0} \int_M k(t,\cdot,q)R_qA_qf\;dM \\
    &=  f(q)
  \end{align*}
  where we use $R_qA_qf\in C^\infty(M)$ in the last equality.
\end{proof}



%
%
%
%

\resettheoremcounters

\section{Radial Eigenfunctions} \label{Eigenfunctions}


In this section we discuss some properties of radially symmetric 
eigenfunctions of the Laplacian in a locally harmonic Blaschke manifold $M$. 
We fix an eigenvalue $\lambda>0$ and write $V^\lambda_p$ for 
the space of radial eigenfunctions around $p\in M$. 
Set $D:=\diam$ and denote by $H:\left]0,D\right[\to\R$ the function with $R_pH=\mean_p$.
Since the linear ODE
$$y'' + Hy' + \lambda y=0$$
is central to this section, we will refer to it as `the ODE'.
The main results are summarised in Proposition \ref{odeprop} and Corollary \ref{eigencor}.
They can also be found in \cite[Section 2]{szabo_main}.

\begin{prop} \label{odeprop}
  The ODE has at exactly one solution $y:\left]0,D\right[\to \R$ with the initial conditions 
  $$\lim_{r\to 0} y(r)=1 \;\;\text{and}\;\; \lim_{r\to 0} y'(r)=0.$$
  This solution can be extended to a smooth function $\Phi_\lambda:[0,D]\to \R$. For every $p\in M$ the function 
  $R_p\Phi_\lambda$ is smooth and for $\varphi\in V_p^\lambda$ it holds $\varphi=\varphi(p)R_p\Phi_\lambda$.
\end{prop}

\begin{proof}
  \emph{Uniqueness:}
  Given two solutions $y_1,y_2:\left]0,D\right[\to\R$ with 
  $$\lim_{r\to 0} y_i(r)=1 \;\;\text{and}\;\; \lim_{r\to 0} y_i'(r)=0, \;\;\; i=1,2$$
  we get a solution $\overline y:=y_1-y_2$ with 
  $$\lim_{r\to 0} \overline y(r)=0 \;\;\text{and}\;\; \lim_{r\to 0} \overline y'(r)=0.$$
  We have to show that $\overline y=0$. By multiplying the ODE with $\overline y'$ we get
  $$0=\overline y''\overline y' + H(\overline y')^2 + \lambda \overline y \,\overline y'=\frac{((\overline y')^2)'}{2}+H(\overline y')^2 +\lambda\frac{(\overline y^2)'}{2}.$$
  By setting 
  $$z:=\frac{1}{2}((\overline y')^2+\lambda \overline y^2)\ge 0$$
  we get 
  $$z'=\frac{1}{2}((\overline y')^2+\lambda \overline y^2)'=-H(\overline y')^2\le 0$$
  on $]0,\eps[$ with $\eps>0$ sufficiently small. Because of
  $$\lim_{r\to 0} z(r)=0$$
  it follows that $z|_{]0,\eps[}=0$ and $\overline y|_{]0,\eps[}=0$. 
  Then $\overline y=0$ holds by the Picard-Lindel\"{o}f theorem.

  \emph{Existence:}
  Let be $\varphi,\psi\in V^\lambda$ and $p,q\in M$. The function $R_pA_p\varphi$ is again 
  an eigenfunction for the eigenvalue $\lambda$ by the global 
  basic commutativity (Theorem \ref{basiccom2}). In particular, 
  $R_pA_p\varphi$ is smooth. 
  For $p$ we pick $\varphi$
  such that $\varphi(p)\neq 0$ and set 
  $$\Phi_\lambda := \frac{A_p\varphi}{\varphi(p)}.$$
  This definition is independent of the choices since by Lemma \ref{lapofradlem} 
  we get that $A_p\varphi$, $A_q\varphi$ and $A_p\psi$ solve the ODE. 
  Hence the claim follows.
\end{proof}

\begin{rem}
  In the following, we will use the notation $\Phi_\lambda:[0,D]\to\R$ for the unique
  extended solution of the ODE with the described initial conditions and call it `the solution'.
\end{rem}

\begin{cor} \label{eigencor}
  The space of eigenfunctions is spanned by the radial eigenfunctions, i.e.
  $$V^\lambda=\spn{V_p^\lambda \;|\; p \in M}=\spn{R_p\Phi_\lambda \;|\; p \in M}.$$
\end{cor}

\begin{proof}
  Assume there were a $0\neq \varphi\in V^\lambda$ with
  $\skl{\varphi}{R_p\Phi_\lambda}=0$ for all $p\in M$. 
  Hence
  $$0=\skl{\varphi}{R_p\Phi_\lambda} = \skl{R_pA_p\varphi}{R_p\Phi_\lambda}= 
  \varphi(p)\skl{R_p\Phi_\lambda}{R_p\Phi_\lambda}.$$
  So either $\varphi=0$ or $\nl{R_p\Phi_\lambda}=0$ for a $p\in M$. Both possibilities 
  contradict the assumptions.
\end{proof}




\begin{prop}[harmonicity and $L^2$-product] \label{conharprop}
  Let $M$ be a locally harmonic Blaschke manifold. Then
  for every $p\in M$ 
  and smooth $F,G:[0,D]\to \R$ the function $$M\to \R, \; q\mapsto \skl{R_pF}{R_qG}$$
  is radial around $p$, i.e.~the $L^2$-product of two radial functions is radial again.
\end{prop}

\begin{proof}
  Let $q\in M$ and $\Omega:[0,D]\to\R$ the function with $R_q\Omega = \den_q$.
  Denote by $(\lambda_i)_{i\in\N_0}$ the spectrum of the Laplacian.
  Then $(R_p\Phi_{\lambda_i})_{i\in\N}$ forms an orthogonal basis of 
  the space of radial functions around $p$.
  Let $a_i\in\R$ be the coefficients of $R_pF$ in this basis.
  
  By Proposition \ref{odeprop} we get
  $$A_qR_p\Phi_{\lambda_i}=(A_qR_p\Phi_{\lambda_i})(0)\Phi_{\lambda_i}=(R_p\Phi_{\lambda_i})(q)\Phi_{\lambda_i}.$$
  Hence
  \begin{align*}
    \skl{R_pF}{R_qG} &= \int_MR_pFR_qG \;dM \\
    &= \sum_{i\in\N_0} a_i\int_M R_p\Phi_{\lambda_i} R_qG \;dM \\
    &= \sum_{i\in\N_0} a_i \int_0^D \int_{S_1(0_q)} (R_p\Phi_{\lambda_i})(\exp_qr\theta) G(r)r^{n-1}\Omega(r) \;d\theta dr\\
    &= \sum_{i\in\N_0} a_i \volS \int_0^D (A_qR_p\Phi_{\lambda_i})(r) G(r)r^{n-1}\Omega(r) \;dr \\
    &= \volS \left(\sum_{i\in\N_0} \left(a_i  \int_0^D \Phi_{\lambda_i}(r) G(r)r^{n-1}\Omega(r) \;dr\right) (R_p\Phi_{\lambda_i})(q)\right).
  \end{align*}
  This implies the claim.
\end{proof}

\begin{rem}
  If we set $F:= \Phi_\lambda=:G$ in the above computation, we get
  $$\skl{R_p\Phi_\lambda}{R_q\Phi_\lambda} = \left(\volS \int_0^D \Phi_\lambda(r)^2 r^{n-1}\Omega(r) \;dr\right) (R_p\Phi_\lambda)(q).$$
  The statement ``if in a Blaschke manifold $M$ the $L^2$-product 
  of two radial functions is radial again, then $M$ is 
  locally harmonic'' is also true, cf. \cite[Proposition 2.1]{szabo_main}.
\end{rem}

\resettheoremcounters

\section{The `Nice Embedding' of Harmonic Manifolds}\label{embedding}

For this section let $M$ be a locally harmonic Blaschke 
manifold. 
The density function $\den_p$
in $p\in M$ is radial with $\den_p=R_p\Omega$ for a suitable
$\Omega:[0,D]\to\R$. 
For a smooth $G:[0,D]\to\R$ we set
$$\nd{G} := \sqrt{\int_0^D G(r)^2r^{n-1}\Omega(r) \;dr}.$$
Then it holds
$$\nl{R_pG}=\sqrt{\volS\int_0^D G(r)^2r^{n-1}\Omega(r) \;dr}=\sqrt{\volS}\nd{G}.$$
The following results allow us to embed $M$ in a Euclidean space such that the 
geodesics are mapped into congruent screw lines. Together with Lemma \ref{circlesymlem} this 
forms the key idea for the proof of Lichnerowicz's conjecture.
The finite-dimensional version can be found in \cite[Theorem 6.99]{besse_closed}, 
the infinite-dimensional in \cite[Theorem 3.1]{szabo_main}.

\begin{thm}[embedding theorem]
  For a non-constant $G\in C^\infty([0,D])$ we define the map
  $$R^G:M\to L^2(M),\;\; p\mapsto R^G(p):=c_GR_pG$$
  with 
  $$c_G:=\frac{\sqrt{n}}{\nd{G'}\sqrt{\volS}}.$$ 
  This map has the following properties.
  \begin{enumerate}
    \item $R^G(M)\subset S_{C_G}$ where $S_{C_G}$ is the sphere in $L^2(M)$ of
      radius 
      $$C_G:=\frac{\nd{G}\sqrt{n}}{\nd{G'}}.$$
    \item For a normalised geodesic $\gamma$ of $M$ the curve $R^G\circ \gamma$ is 
      a screw line of $L^2(M)$. For two normalised geodesics $\gamma$ and $\sigma$ of $M$ 
     the screw lines $R^G\circ \gamma$ and $R^G\circ \sigma$ have the same screw function.
     They are therefore congruent.
    \item $R^G$ is an isometric immersion.
  \end{enumerate}
\end{thm}

\begin{proof}
  \begin{itemize}
    \item[\emph{1.}]
      For $p\in M$ we have
      $$\nl{R^G(p)} = c_G\nl{R_pG} = c_G\sqrt{\volS}\nd{G} = C_G.$$
      This means $R^G(M)\subset S_{C_G}$.
    \item[\emph{2.}]
      For $p,q\in M$ we have
      \begin{align*}
	\nl{R^G(p) - R^G(q)}^2 &= \nl{R^G(p)}^2+ \nl{R^G(q)}^2- 2 \skl{R^G(p)}{R^G(q)}\\
	&= 2C_G^2-2 \skl{R^G(p)}{R^G(q)} \\
	&= 2C_G^2-2 c_G^2\skl{R_pG}{R_qG}.
      \end{align*}
      By Proposition \ref{conharprop} the function $\skl{R_pG}{R_qG}$ only
      depends on $d(p,q)$. For $s_0,s\in\R$ we set $p:=\gamma(s_0+s)$ and 
      $q:=\gamma(s_0)$ respectively $p:=\sigma(s_0+s)$ and 
      $q:=\sigma(s_0)$ to get the claim.
    \item[\emph{3.}]
    Pick $p\in M$ and $v\in T_pM$ with $\n{v}=1$. Let $\gamma$ be a geodesic parametrised 
      by arc length with $\gamma(0)=p$ and $\gamma'(0)=v$. We have 
      \begin{align*}
	\nl{(dR^G)_p(v)} &= \nl{\left.\frac{d}{dt}\right|_{t=0} R^G(\gamma(t))} \\
	&= c_G\nl{\left.\frac{d}{dt}\right|_{t=0} R_{\gamma(t)}G} \\
	&= c_G\nl{\left.\frac{d}{dt}\right|_{t=0} G(d(\gamma(t),\cdot))} \\
	&= c_G\sqrt{\int_M \left(\left.\frac{d}{dt}\right|_{t=0} G(d(\gamma(t),\cdot))\right)^2\;dM} \displaybreak\\
	&= c_G\sqrt{\int_0^D\int_{S_1(0_p)} \left(\left.\frac{d}{dt}\right|_{t=0} G(d(\gamma(t),\exp_pr\theta))\right)^2 r^{n-1}\Omega(r) \;d\theta dr} \\
	&= c_G\sqrt{\int_0^D\int_{S_1(0_p)} G'(d(p,\exp_pr\theta))^2 \; \cos^2 \angle(v,\theta) \; r^{n-1}\Omega(r) \;d\theta dr} \\
	&= c_G\sqrt{\int_0^D G'(r)^2 r^{n-1}\Omega(r)\;dr} \sqrt{\int_{S_1(0_p)} \cos^2 \angle(v,\theta)  \;d\theta} \\
	&= c_G\nd{G'}\sqrt{\frac{\volS}{n}}\\
	&= 1.
      \end{align*}
      This shows that $R^G$ is an isometric immersion.
  \end{itemize}
\end{proof}

\begin{cor}[Besse's nice embedding: special case $G=\Phi_\lambda$] \label{bessenicecor}
  For an eigenvalue $\lambda>0$ of the Laplacian denote by $\Phi:=\Phi_\lambda$ the 
  solution of the ODE and
  set $\overline M:=R^{\Phi}(M)$.
  \begin{enumerate}
    \item Let $\Phi(D)=1$ and $M$ be diffeomorphic 
      to the sphere $S^n$. 
      Then $\overline M$ is diffeomorphic to $\RP^n$ and a locally harmonic Blaschke manifold. 
      The map 
      $R^{\Phi}:M\to \overline M$ is the universal Riemannian covering map.
    \item Let $\Phi(D)\neq 1$ or $M$ be not diffeomorphic 
      to the sphere $S^n$. 
      Then the map $R^{\Phi}:M\to V^\lambda$ is an 
      injective isometric immersion, i.e.~an embedding since $M$ is compact.
      The manifold $\overline M$ is a minimal submanifold of the sphere $S_{C_{\Phi}}$.
      For a unit speed geodesic $\gamma$ of $M$ set $c:=R^{\Phi}\circ \gamma$. 
      Then we have for every $s_0,s\in\R$
      $$\skl{c(s_0)}{c(s)} = C_{\Phi}^2\Phi(d(\gamma(s_0),\gamma(s))).$$
  \end{enumerate}
\end{cor}

\begin{proof}
  Let $p,q\in M$ be points with $R^{\Phi}(p)=R^{\Phi}(q)$.
  From the remark after Proposition \ref{conharprop} and the proof of the second 
  statement of the embedding theorem we get
  \begin{align*}
    0 = \nl{R^{\Phi}(p) - R^{\Phi}(q)}^2 &= 2C_{\Phi}^2-2c_{\Phi}^2  \skl{R_p\Phi}{R_q\Phi} \\
    &= 2C_{\Phi}^2 -  2\frac{n}{\nd{\Phi'}^2\volS} \volS \nd{\Phi}^2  (R_p\Phi)(q)\\
    &= 2C_{\Phi}^2 - 2C_{\Phi}^2 (R_p\Phi)(q).
  \end{align*}
  It follows $1 = (R_p\Phi)(q) = \Phi(d(p,q))$.
  This means that $R^{\Phi}(p)=R^{\Phi}(\overline q)$ 
  for all $\overline q\in S_{d(p,q)}^d(p)$. 
  We recall that $S_{d(p,q)}^d(p)$ is a submanifold of $M$. 
  But then it must be a single point since otherwise we had a contradiction 
  to the fact that $R^{\Phi}$ is an isometric immersion.
  The only case in which $S_{d(p,q)}^d(p)$ is singleton occurs for
  $M$ diffeomorphic to the sphere $S^n$ and $d(p,q)=D$, cf. Proposition \ref{blaschkespecial}. 
  Then $\overline M$ is Blaschkean and diffeomorphic to $\R\text{P}^n$.
  The map $R^{\Phi}:M\to \overline M$ is the universal 
  Riemannian covering map and therefore $\overline M$ locally harmonic. 
  This completes the first part. 

  Now we can assume that $R^\Phi$ is injective.
  With the formula in the remark after Proposition \ref{conharprop} we compute for $s_0,s\in\R$
  \begin{align*}
    \skl{c(s_0)}{c(s)} &= c_{\Phi}^2  \skl{R_{\gamma(s_0)}\Phi}{R_{\gamma(s)}\Phi}
    = c_{\Phi}^2 \volS \nd{\Phi}^2 \Phi(d(\gamma(s_0),\gamma(s)))\\
    &=C_{\Phi}^2 \Phi(d(\gamma(s_0),\gamma(s))).
  \end{align*}
  We are left to show that the embedding is minimal. First we remark that for every $p\in M$ 
  \begin{align*}
    \lambda &= \frac{\skl{\Delta R_p\Phi}{R_p\Phi}}{\nl{R_p\Phi}^2} 
    = \frac{1}{\nl{R_p\Phi}^2} \int_M \skl{\grad R_p\Phi}{\grad R_p\Phi} \;dM \\
    &= \frac{1}{\nl{R_p\Phi}^2} \int_M \nabla_{E^p} R_p\Phi \; \nabla_{E^p} R_p\Phi \;dM 
    = \frac{\nl{R_p\Phi'}^2}{\nl{R_p\Phi}^2} 
    = \frac{\nd{\Phi'}^2}{\nd{\Phi}^2} \\
    &= \frac{n}{C^2_{\Phi}}
  \end{align*}
  holds.
  Set $N:=\dim V^\lambda$ and choose an $L^2$-orthonormal basis $(\varphi_1, \dots,\varphi_N)$ of $V^\lambda$. 
  Coordinates $(x_1,\dots,x_N)$ on $\overline M$ are given by
  $$x_i(R^{\Phi}(p)) := \skl{\varphi_i}{R^{\Phi}(p)}=c_{\Phi}\int_M \varphi_i R_p\Phi \;dM, \;\;i=1,\dots,N.$$
  The submanifold $\overline M\subset S_{C_{\Phi}}$ is minimal if and only if every $x_i$ is 
  an eigenfunction to the eigenvalue $\frac{n}{C^2_{\Phi}}$, cf. \cite[Note 14, Example 3]{kobayashi_geo}. Because of 
  $\Delta R_p\Phi=\lambda R_p\Phi$ this is equivalent to $\lambda = \frac{n}{C^2_{\Phi}}$.



\end{proof}

\begin{rem}
  Since we show in the next section that a locally harmonic Blaschke manifold which is
  diffeomorphic to $\R\text{P}^n$ carries the canonical metric, our $M$ in the first case is then 
  the sphere with the canonical metric. Hence we need not consider the first case in the following considerations.
  
  Noteworthy is the characterisation of globally harmonic manifolds and Blaschke manifolds 
  through (minimal) embeddings into a sphere such that all geodesics are mapped into 
  congruent screw lines, cf. \cite[Theorems 6.2 and 6.5]{sakamoto_helical}.

  The embedding in the second case above is actually Besse's
  nice embedding, cf. \cite[Theorem 6.99]{besse_closed}.
  It is defined by 
  $$M\ni p\mapsto \sqrt{\frac{n\vol(M)}{\lambda N}}\left(\varphi_1(p), \dots, \varphi_N(p)\right)\in \R^N.$$
  We have for every $p\in M$ and $i=1,\dots,N$
  $$\skl{R^\Phi(p)}{\varphi_i} = c_\Phi \volS\nd{\Phi}^2\varphi_i(p)$$
  and therefore
  $$\sqrt{\frac{n\vol(M)}{\lambda N}} = c_\Phi \volS\nd{\Phi}^2= C_\Phi \nl{R_p\Phi}=\sqrt{\frac{n}{\lambda}}\nl{R_p\Phi}$$
  or
  $$\vol(M) =N\nl{R_p\Phi}^2.$$
\end{rem}
  


\resettheoremcounters

\section{Proof of Lichnerowicz's Conjecture}

In this section let $M$ be a locally harmonic Blaschke manifold and
assume without loss of generality that $\diam=\pi$. 
By pinning down the possible density functions of $M$ (Lemma \ref{formofden}) 
 we are able to find its first eigenvalue and to solve the ODE for it (Lemma \ref{firsteigenlem}). 
Then we present two variants of the proof of Lichnerowicz's conjecture. 
The first one uses the nice embedding (Corollary \ref{bessenicecor}) and Lemma \ref{circlesymlem}. 
The second one is intrinsic, but more complex so that we only 
refer to the literature.

For the rest of the section we fix an eigenvalue $\lambda>0$, a point $p\in M$ and 
the solution $\Phi:=\Phi_\lambda$ of the ODE.
From now on we consider the average $A_pf:[0,\pi]\to\R$ of a radial function $f:M\to\R$ 
around $p$ to be periodically extended to $\R$. That means we consider the function
$f\circ \gamma:\R\to\R$, where $\gamma:\R\to M$ is a unit speed geodesic with $\gamma(0)=p$,
instead of $A_pf:[0,\pi]\to\R$. This new function is $2\pi$-periodic 
and even. In particular, $\Phi:\R\to\R$ has these properties.
Alternatively, we can set 
$$A_pf:\R\to\R,\;\; r\mapsto A_pf\left(\pi - \left|\pi - \left|r\right| \bmod 2\pi\right|\right)$$
since 
$$\forall \;r,t\in\R:\;\;d(\gamma(r),\gamma(t)) = \pi - \left|\pi - \left|r-t\right| \bmod 2\pi\right|$$
holds.
Furthermore we set $\Omega:=A_p\den_p$ and 
$$\hat\Omega:\R\to\R,\;\; r \mapsto r^{n-1}\Omega(r)$$
so that in particular $\hat\Omega$ is odd, $\hat\Omega^2$ is even and
$$\Phi''+\frac{\hat\Omega'}{\hat\Omega}\Phi'+\lambda\Phi=0.$$
holds on $\R\setminus \{k\pi\;|\;k\in\Z\}$.


\subsection{Possible Density Functions}

We present Szab\'o's careful analysis of the possible 
forms of density functions for locally harmonic Blaschke manifolds.
More precisely, our aim is it to show Lemma \ref{formofden}, which states that 
the function $\hat\Omega$ is the product of a power of sine and a power
of cosine. 
We follow \cite[Section 4]{szabo_main} with two exceptions. The proof of 
Lemma \ref{denpolylem} is a slightly changed version of \cite[Theorem 2]{nikolayevsky_harm} and the proof of 
Lemma \ref{rootdistrilem} is new. \\

First we show that $\Phi$ and $\hat \Omega^2$
are trigonometric polynomials of a special form.

\begin{lem} \label{phipolylem}
  There is a polynomial $P:\R\to\R$ with real coefficients 
  such that $$\Phi = P \circ \cos.$$
\end{lem}

\begin{proof}
  Let $\gamma:\R\to M$ be a unit speed geodesic in $M$ with $\gamma(0)=p$.
  We have 
  \begin{align*}
    \spn{(R_q\Phi)\circ\gamma\;|\; q\in \gamma(\R)} &= \spn{\Phi(d(\gamma(\cdot),q))\;|\; q\in \gamma(\R)}\\
    &= \spn{\Phi(\pi - \left|\pi - \left|\cdot - t\right| \bmod 2\pi\right|)\;|\; t\in\R} \\
    &= \spn{\Phi(\cdot-t)\;|\; t\in \R}.
  \end{align*}
  Since 
  $\spn{R_q\Phi\;|\; q\in \gamma(\R)}$ is a subspace of the 
  finite-dimensional $V^\lambda$, it is finite-dimensional. 
  Because precomposing with $\gamma$
  is linear, we have that 
  $\spn{\Phi(\cdot-t)\;|\; t\in \R}$
  is a finite-dimensional subspace of $C^\infty(\R)$.
  Because $\Phi$ is $2\pi$-periodic and even, the claim follows 
  from the Lemmata \ref{specialspanlem} and \ref{chebpoly}.
\end{proof}

%
%



\begin{lem} \label{denpolylem}
  There is a polynomial $O:\R\to\R$ with real coefficients 
  such that 
  $$\hat \Omega^2 = O \circ \cos.$$
\end{lem}

\begin{proof}
  Let $\gamma:\R\to M$ be a unit speed geodesic in $M$ with $\gamma(0)=p$ and 
  let $(e_2,\dots,e_n)$  be a positively oriented orthonormal basis
  of $T_p^\perp\gamma$. Denote by $(E_2,\dots,E_n)$ its parallel translates along 
  $\gamma$. 
  In this proof we will use the representation of Jacobi tensors in
  the basis $(E_2,\dots,E_n)$, i.e.~they are considered to be maps 
  $\R\to \R^{(n-1)\times(n-1)}$.

  Denote by $J$ and $K$ the Jacobi tensors along $\gamma$ with initial conditions
  $J(0)=0$, $J'(0)=\I$, $K(0)=\I$ and $K'(0)=0$ where $\I\in \R^{(n-1)\times(n-1)}$ is the 
  identity matrix.
  Let $r\in\R$ and $t\in \R\setminus \{k\pi\;|\;k\in\Z\}$. We set
  $$L(t):=J^{-1}(t) K(t)$$
  and
  $$\mathcal{J}(t):=J'(t)J^{-1}(t) K(t)-K'(t)=J'(t) L(t)-K'(t).$$ 
  Because of 
  $$J^T(t)J'(t)-(J^T)'(t)J(t)=0$$
  and
  $$J^T(t)K'(t)-(J^T)'(t)K(t)=-\I$$
  we get
  $$J^T(t)\mathcal{J}(t) = J^T(t)J'(t)J^{-1}(t)K(t)-J^T(t)K'(t) = (J^T)'(t)K(t)-J^T(t)K'(t) = \I.$$
  Hence $\mathcal{J}(t)$ is invertible with $\det \mathcal{J}^{-1}(t) = \det J^T(t) = \det J(t)= \hat\Omega(t)$. 

  Set
  $$J_t(r):=(J(r)L(t)-K(r)) \mathcal{J}^{-1}(t).$$
  Because $J_t$ is a Jacobi tensor along $\gamma$ with 
  $$J_t(t) = (J(t)L(t)-K(t))\mathcal{J}^{-1}(t) = 0$$
  and 
  $$J_t'(t) = (J'(t)L(t)-K'(t))\mathcal{J}^{-1}(t) = \I$$
  it holds $\det J_t(r)=\hat\Omega(r-t)$. Hence
  $$\hat\Omega(r-t) = \det J_t(r)=\det\left(J(r)L(t)-K(r)\right) \det \mathcal{J}^{-1}(t)= \det\left(J(r)L(t)-K(r)\right)\hat\Omega(t)$$
  and
  $$\hat\Omega^2(r-t)= \det\left(J(r)L(t)-K(r)\right)^2\hat\Omega^2(t).$$

  By expanding the determinant we see that $\spn{\hat\Omega^2(\cdot - t)\;|\; t\in\R\setminus \{k\pi\;|\;k\in\Z\}}$ 
  is finite-dimensional and therefore $\spn{\hat\Omega^2(\cdot - t)\;|\; t\in\R}$ as well. 
  The Lemmata \ref{specialspanlem} and \ref{chebpoly} yield the claim.



%
%

\end{proof}

The next step is to examine $P$ and $O$ by 
finding restrictions to their possible roots.

%

\begin{lem}
  The numbers $-1$ and $1$ are roots of $O$.
\end{lem}

\begin{proof}
  This follows from 
  $$O(-1)=O(\cos \pi)=\hat \Omega(\pi)^2=\pi^{2n-2}\Omega(\pi)^2=0$$
  and 
  $$O(1)=O(\cos 0)=\hat \Omega(0)^2=0\cdot\Omega(0)^2=0.$$
\end{proof}

\begin{lem}
  The following three statements hold. 
  \begin{enumerate}
    \item All roots of $P$ have multiplicity one.
    \item All roots of $P'$ have multiplicity one.
    \item Except $-1$ and $1$, all roots of $O$ are also roots of $P'$. 
  \end{enumerate}
\end{lem}

%
%

\begin{proof}
  In $\R\setminus \{k\pi\;|\;k\in\Z\}$ we have the equality
  $$\Phi''+\frac{\hat\Omega'}{\hat\Omega}\Phi'=-\lambda\Phi.$$
  In the first part of the proof we work in a compact interval of $\R\setminus\{k\pi\;|\;k\in\Z\}$ 
  where $\Phi'$ has no roots.
  By setting 
  $$Q:=O(P')^2(1-\id^2)$$ 
  we get
  $$Q\circ \cos=(O\circ \cos)(P'\circ \cos)^2(1-\cos^2)=(O\circ \cos)(P'\circ \cos)^2\sin^2
  =\hat\Omega^2(\Phi')^2$$
  and 
  \begin{align*}
    \left(\log (Q\circ \cos)\right)' &= \left(\log \left(\hat\Omega^2\left(\Phi'\right)^2\right)\right)' 
    = \frac{\left(\hat\Omega^2\left(\Phi'\right)^2\right)'}{\hat\Omega^2(\Phi')^2} 
    = \frac{(\hat\Omega^2)'}{\hat\Omega^2}+\frac{\left(\left(\Phi'\right)^2\right)'}{(\Phi')^2}
    = 2\left(\frac{\hat\Omega'}{\hat\Omega}+\frac{\Phi''}{\Phi'}\right) 
    = -2\lambda\frac{\Phi}{\Phi'} \\
    &= 2\lambda\frac{P\circ \cos}{(P'\circ \cos)\sin}.
  \end{align*}
  Hence
  $$\log (Q\circ \cos) = 2\lambda \int \frac{P\circ \cos}{(P'\circ \cos)\sin}$$
  and the substitution of $\cos$ yields
  $$\log Q = -2\lambda \int \frac{P}{(1-\id^2)P'}.$$

  Let be $x\in\R$ for the rest of the proof.
  Let $\pi_1,\dots,\pi_\nu\in\C$ be the (distinct) roots of $P$ with multiplicities
  $p_1,\dots,p_\nu$. Denote by $\varrho_1,\dots,\varrho_\mu\in\C$ the (distinct) roots of $P'$ which are
  not roots of $P$ and by $r_1,\dots,r_\mu$ their multiplicities. Let the leading 
  coefficients be $A$ and $B$ respectively.
  We can write 
  $$P(x)=A(x-\pi_1)^{p_1}\cdots (x-\pi_\nu)^{p_\nu},$$
  $$P'(x)=B(x-\pi_1)^{p_1-1}\cdots (x-\pi_\nu)^{p_\nu-1}(x-\varrho_1)^{r_1}\cdots (x-\varrho_\mu)^{r_\mu}$$
  and
  $$\log Q(x) = \frac{-2\lambda A}{B} \int \frac{(x-\pi_1)\cdots (x-\pi_\nu)}{(1-x)(1+x)(x-\varrho_1)^{r_1}\cdots (x-\varrho_\mu)^{r_\mu}}\,dx.$$
  By the partial fraction expansion of the integrand we get that 
  $r_1=\dots=r_\mu=1$ and $-1\neq\varrho_i\neq 1$ for $i=1,\dots,\mu$ 
  since otherwise $Q$ would not be a polynomial.
  Moreover the partial fraction expansion gives us
  $$Q(x)=C(1-x)^\sigma(1+x)^\tau(x-\varrho_1)^{q_1}\cdots (x-\varrho_\mu)^{q_\mu}$$
  where $\sigma,\tau,q_1,\dots,q_\mu \in\N_0$ and $C\in\R$. By the definition of $Q$ we even know 
  $\sigma,\tau\ge 1$ and $q_1,\dots,q_\mu\ge 2$.

  Since $O$ is a polynomial and 
  \begin{align*}
    O(x) = {}& Q(x)(P')^{-2}(x)(1-x^2)^{-1} \\
    = {}& C(1-x)^\sigma(1+x)^\tau(x-\varrho_1)^{q_1}\cdots (x-\varrho_\mu)^{q_\mu} \\
    & \cdot\; B^{-2}(x-\pi_1)^{-2(p_1-1)}\cdots (x-\pi_\nu)^{-2(p_\nu-1)}(x-\varrho_1)^{-2}\cdots (x-\varrho_\mu)^{-2}\\
    & \cdot\; (1-x^2)^{-1} \\
    = {}& CB^{-2}(1-x)^{\sigma-1}(1+x)^{\tau-1}(x-\varrho_1)^{q_1-2}\cdots (x-\varrho_\mu)^{q_\mu-2} \\
    & \cdot\;(x-\pi_1)^{-2(p_1-1)}\cdots (x-\pi_\nu)^{-2(p_\nu-1)}
  \end{align*}
  holds,
  we get $-2p_i+2 \ge 0$ for $i=1,\dots,\nu$ and therefore $p_1=\dots=p_\nu=1$.
\end{proof}

We keep the notation of the above proof,
i.e.~denote by 
$\pi_1,\dots,\pi_\nu$ the roots of $P$ and by
$\varrho_1,\dots,\varrho_{\nu-1}$ the roots of $P'$.
Then the roots of $O$ are contained in $\{-1,1,\varrho_1,\dots,\varrho_{\nu-1}\}$.

%
%

\begin{lem} \label{rootdistrilem}
  The roots of $P$ and $P'$ are real numbers and if we arrange them in ascending order, it holds
  $$-1 < \pi_1 < \varrho_1 < \pi_2 < \dots < \pi_{\nu-1} < \varrho_{\nu-1} < \pi_\nu < 1.$$
\end{lem}

\begin{proof}
  From the above proof we have 
  \begin{align*}
    (-\sin) \left(O(P')^2\left(1-\id^2\right)\right)'\circ \cos 
    &= \left(\left(O(P')^2\left(1-\id^2\right)\right)\circ \cos\right)'
    = \left(\hat\Omega^2 \left(\Phi'\right)^2\right)' 
    = -2\lambda \hat\Omega^2 \Phi\Phi'\\
    &= -2\lambda (-\sin)(O P P')\circ\cos.
  \end{align*}
  Lemma \ref{lucasthm} implies that the roots of $O P P'$ lie 
  in the convex hull of the roots of $O(P')^2(1-\id^2)$, i.e.
  $$\{-1,1,\pi_1,\dots,\pi_\nu,\varrho_1,\dots,\varrho_{\nu-1}\} \subset \conv{-1,1,\varrho_1,\dots,\varrho_{\nu-1}}.$$
  From this we get
  $$\conv{-1,1,\pi_1,\dots,\pi_\nu}\subset \conv{-1,1,\varrho_1,\dots,\varrho_{\nu-1}}.$$
  Because of
  $$\{\varrho_1,\dots,\varrho_{\nu-1}\} \subset \conv{\pi_1,\dots,\pi_\nu}$$
  we have 
  $$\{-1,1,\varrho_1,\dots,\varrho_{\nu-1}\} \subset \conv{-1,1,\pi_1,\dots,\pi_\nu}$$
  and 
  $$\conv{-1,1,\varrho_1,\dots,\varrho_{\nu-1}} \subset \conv{-1,1,\pi_1,\dots,\pi_\nu}.$$
  Since 
  $$\{\varrho_1,\dots,\varrho_{\nu-1}\}\cap \{\pi_1,\dots,\pi_\nu\} = \emptyset$$
  we get
  $$\conv{-1,1,\varrho_1,\dots,\varrho_{\nu-1}} =\conv{-1,1,\pi_1,\dots,\pi_\nu}=[-1,1].$$
  From this the claim follows.
\end{proof}

\begin{lem}
  The polynomial $O$ has no roots other than $-1$ and $1$.
\end{lem}

\begin{proof}
  We prove the lemma by contradiction. Without loss of generality we may assume 
  that $\varrho_1$ is a root of $O$. Since $-1<\varrho_1<1$ by the last lemma,
  there is $0<r_0<\pi$ with $\cos r_0=\varrho_1$. Then
  $\hat\Omega^2(r_0)=O(\cos r_0)=O(\varrho_1)=0$. This is a contradiction.
\end{proof}

We are now in the position to prove the result we were looking for.

\begin{prop} \label{formofden}
  There are $\tilde C,\alpha,\beta\in\R$ such that
  $$\hat\Omega=\tilde C (1-\cos)^\beta \sin^\alpha.$$
\end{prop}

\begin{proof}
  For all $x\in\R$ we can write 
  $$O(x)=C(1-x)^\sigma(1+x)^\tau$$
  with suitable $\sigma,\tau\in\N$ and $C\in\R^{>0}$.
  Then for all $r\in\R$ holds
  \begin{align*}
    \hat\Omega(r) = \sqrt{O(\cos r)} &= \sqrt{C} (1-\cos r)^{\frac{\sigma}{2}} (1+\cos r)^{\frac{\tau}{2}} 
    = \sqrt{C} (1-\cos r)^{\frac{\sigma}{2}-\frac{\tau}{2}} \sin^{\tau}r.
  \end{align*}
\end{proof}


\begin{rem}
  We keep the notation and get for the mean curvature function
  \begin{align*}
    H := \frac{\hat\Omega'}{\hat\Omega} &= \frac{((1-\cos)^\beta \sin^\alpha)'}{(1-\cos)^\beta \sin^\alpha} \\
    &= \frac{\beta(1-\cos)^{\beta-1} \sin^{\alpha+1} + \alpha(1-\cos)^\beta \cos\sin^{\alpha-1}}{(1-\cos)^\beta \sin^\alpha} \\
    &= \frac{\beta\sin^2+\alpha(1-\cos)\cos}{(1-\cos)\sin} \\
    &= \frac{(\alpha+ \beta)\cos+\beta}{\sin}.
  \end{align*}
  Using Proposition \ref{harmeinsteinprop} and after some lengthy calculations
  we compute the Ricci curvature to be $\alpha + \frac{1}{2}\beta$.
  Since $\hat \Omega$ vanishes of order $n-1$ in $0$ we have
  $\alpha + 2\beta =n-1$. Because of $\Omega(0)=1$ we can deduce $\tilde C=2^\beta$. 

  Actually we can say even more.
  By the Bott-Samelson theorem, cf. \cite[Theorem 7.23]{besse_closed}, 
  we know that $\hat \Omega$ vanishes of order $n-1,0,1,3$ or $7$ in $\pi$. Hence $\alpha$ can 
  only take the values $n-1,0,1,3$ or $7$. Then $\beta$ equals $0,\frac{n-1}{2},\frac{n-2}{2},\frac{n-4}{2}$ 
  or $\frac{n-8}{2}$ respectively. If we set $n=m,2m,4m$ 
  or $16$ respectively, we recover the density functions of the ROSSs (Proposition \ref{denofrossprop}).
\end{rem}

\subsection{Spectrum and Radial Eigenfunctions} \label{specnradeigen}

Because of Lemma \ref{formofden} it is now easy to construct 
concrete eigenvalues and radial eigenfunctions of the Laplacian.
We keep the notation of this lemma and additionally set $\lambda_1:=\alpha+\beta+1$.


\begin{lem} \label{firsteigenlem}
  The number $\lambda_1$ is an eigenvalue and  
  $$\Phi:\R\to\R,\;\; r\mapsto \frac{\lambda_1}{\lambda_1+\beta}\left(\cos r+\frac{\beta}{\lambda_1}\right)$$ 
  is the solution of the ODE, i.e.~$\Phi=\Phi_{\lambda_1}$.
\end{lem}

\begin{proof}
  The function $R_p\Phi$ is obviously smooth for every $p\in M$. 
  We have
  $$\Phi'=-\frac{\lambda_1}{\lambda_1+\beta}\sin\;\;\;\text{and}\;\;\; \Phi''=-\frac{\lambda_1}{\lambda_1+\beta}\cos.$$
  The initial conditions $\Phi(0)=1$ and $\Phi'(0)=0$ are satisfied.
  Furthermore 
  %
  \begin{align*}
    \Phi''+\frac{\hat\Omega'}{\hat\Omega}\Phi' +\lambda_1\Phi &= \frac{\lambda_1}{\lambda_1+\beta}(-\cos - (\alpha+\beta)\cos - \beta +\lambda_1\cos + \beta)\\
    &= 0.
  \end{align*}
  This implies the claim.
\end{proof}

\begin{lem}
  Set $\lambda_k:=k(k+\alpha+\beta)$ for $k\in\N$. Then $\lambda_k$ is an eigenvalue and 
  the solutions $\Phi_{\lambda_k}$ of the ODE is given by  
  $$\Phi_{\lambda_k}:\R\to\R,\;\; r\mapsto \sum_{i=0}^k a_i\cos^i r,$$
  with certain $a_i\in\R$.
  The spectrum of $M$ is $(\lambda_k)_{k\in\N_0}$.
\end{lem}

\begin{proof}
  Let $k\in\N$. The function $R_p\Phi_{\lambda_k}$ is obviously smooth for every $p\in M$. We have
  $$\Phi'_{\lambda_k}=-\sin\sum_{i=0}^k i a_i  \cos^{i-1},$$
  $$\Phi''_{\lambda_k}=-\sum_{i=0}^k i a_i \cos^i+(1-\cos^2)\sum_{i=0}^k i(i-1) a_i \cos^{i-2}
  =-\sum_{i=0}^k i^2 a_i \cos^i + \sum_{i=-2}^{k-2} (i+2)(i+1) a_{i+2} \cos^i$$
  and 
  $$\frac{\hat\Omega'}{\hat\Omega}\Phi'_{\lambda_k}=-((\alpha+\beta)\cos+\beta)\sum_{i=0}^k i a_i \cos^{i-1}
  =-(\alpha+\beta)\sum_{i=0}^k i a_i \cos^i -\beta\sum_{i=-1}^{k-1} (i+1) a_{i+1} \cos^i.$$
  Hence
  \begin{align*}
    0 &= \Phi''_{\lambda_k}+\frac{\hat\Omega'}{\hat\Omega}\Phi'_{\lambda_k} +\lambda_k\Phi_{\lambda_k} \\
    &= \sum_{i=0}^k ( (k^2-i^2+(k-i)(\alpha+\beta))a_i + (-\beta i -\beta)a_{i+1} + (i^2+3i+2)a_{i+2} )\cos^i
  \end{align*}
  where we set $a_{k+2}:=0=:a_{k+1}$. 
  Since $k^2-i^2+(k-i)(\alpha+\beta)\neq 0$ for $i\neq k$ we get a recursive formula for the $a_i$ 
  if we require $\Phi_{\lambda_k}(0)=1=\sum_{i=0}^ka_i$.
  Because $(\Phi_{\lambda_k})_{k\in\N}$ spans the space consisting of all polynomials in cosine, $(\lambda_k)_{k\in\N_0}$ is 
  the whole spectrum.
\end{proof}

\subsection{Two Variants of the Proof}

We keep the definitions of $\alpha,\beta,\lambda_1$ and $\Phi$ 
from the last section.\\

\textbf{First Variant.} \label{var1}
So far we have not used the embedding at all. In order to be allowed to use 
the second part of Corollary \ref{bessenicecor} we only consider the case where 
$M$ is not diffeomorphic to the sphere $S^n$ in this first variant of the proof.  


\begin{lem} \label{lastlem}
  All geodesics of $R^{\Phi}(M)$ are circles.
\end{lem}

\begin{proof}
  For a unit speed geodesic $c$ in $R^{\Phi}(M)$ we have 
  $$\sklr{c(0)}{c(s)} = C^2_{\Phi} \frac{\lambda_1}{\lambda_1+\beta}\left(\cos s + \frac{\beta}{\lambda_1}\right)$$
  for all $s\in\R$ by the second part of Corollary \ref{bessenicecor}. 
  The screw function $S_0$ of $c$ is therefore
  $$S_0(s)=2C^2_{\Phi}-2C^2_{\Phi}\frac{\lambda_1}{\lambda_1+\beta}\left(\cos s + \frac{\beta}{\lambda_1}\right)
  =2\frac{\lambda_1}{\lambda_1+\beta}C^2_{\Phi}-2\frac{\lambda_1}{\lambda_1+\beta}C^2_{\Phi}\cos s.$$
  Because a circle of radius $\sqrt{\frac{\lambda_1}{\lambda_1+\beta}}C_{\Phi}$ has got 
  the same screw function, $c$ is a circle.
\end{proof}

\begin{rem}
  Taking the proof of Corollary \ref{bessenicecor} and the remark after Proposition \ref{formofden} into account we get that
  $C_{\Phi}^2=\frac{n}{\lambda_1}$ and $\lambda_1+\beta =n$ respectively. Hence the circles are of radius $1$.
\end{rem}

\begin{lem} \label{circlesymlem}
  Let $\overline M$ be the $n$-dimensional submanifold $R^{\Phi}(M)$ of $V^{\lambda_1}$. 
  Then $\overline M$ is a ROSS.
\end{lem}

\begin{proof}
  Fix a point $p\in \overline M$. Denote by $T^\perp_p\overline M$ the normal space of $\overline M$ in $p$. Let 
  $s_p:V^{\lambda_1}\to V^{\lambda_1}$ be the reflection at the affine subspace
  $T^\perp_p\overline M$. 
  For a geodesic $c:\R\to \overline M$ of $\overline M$ with $c(0)=p$ we have
  $s_p(c(0))=p$, $s_p(c'(0))=-c'(0)$ and 
  $s_p(c''(0))=c''(0)$. Since a circle is determined by 
  this data, we have $s_p(c(\R))=c(\R)$. In particular, 
  it holds $s_p(\overline M)=\overline M$.
  Since $s_p$ is an isometry of $V^{\lambda_1}$, it is one of $\overline M$.
  This shows that $\overline M$ is a Riemannian symmetric space.
  If it were not of rank $1$, it would have non-closed geodesics in
  maximal flats.
\end{proof}

\textbf{Second Variant.} \label{var2}
The second variant is an intrinsic proof, which uses \cite[Theorem 1]{ranjan_first}.
Since the averaged eigenfunction $\Phi$ has got no saddle point, we only have to
check that equality holds in Ros's estimate for the first eigenvalue, cf. \cite[Theorem 4.2]{ros_first}.
Equality holds because of
$$\lambda_1 = \alpha+\beta+1= n-1-2\beta+\beta+1 = n-\beta$$
and 
$$\frac{1}{3}(2\ric + n+2) = \frac{1}{3}(2\alpha + \beta + n+2)= \frac{1}{3}(2n-2 -4\beta + \beta + n+2) = n-\beta.$$

\newpage

\appendix

\resettheoremcounters

\section{Appendix}

All the auxiliary results are collected here.

\begin{lem} \label{spanlem}
  Let $F:\R\to\R$ be smooth. The following statements are equivalent.
  \begin{enumerate}
    \item The vector space 
      $$V:=\spn{F(\cdot-t)\;|\;t\in\R}\subset C^\infty(\R)$$ 
      is of finite dimension.
    \item The function $F$ solves a linear ODE with constant coefficients.
    \item There are $k\in\N$, $\alpha_i,\beta_i\in\R$ and polynomials $P_i,Q_i:\R\to\R$ with real coefficients such that
      $$\forall\; x\in\R: \;\; F(x) = \sum_{i=1}^k (P_i(x)\sin \beta_i x + Q_i(x)\cos \beta_i x)e^{\alpha_i x}.$$
  \end{enumerate}
\end{lem}

\begin{proof}
  $\emph{1.} \Rightarrow \emph{2.:}$
  For every $t\in\R$ the map 
  $$B_t:V\to V,\;\; G\mapsto B_tG:=G(\cdot -t)$$
  is an endomorphism of $V$. Furthermore $(B_t)_{t\in\R}$ is a smooth one-parameter
  subgroup of $\operatorname{End}(V)$. So there is $B\in\operatorname{End}(V)$ with
  $$B_t=\exp (tB).$$
  We have for all $x\in\R$
  \begin{align*}
    F'(x) &= \partial_x\left(\left(B_0F\right)(x)\right) 
    = \partial_x\left(\left(B_xF\right)(0)\right) 
    = \partial_x\left(\left(\exp\left(xB\right) F\right)(0)\right) \\
    &= \left(B\left(\exp\left(xB\right) F\right)\right)\left(0\right) 
    = \left(B\left(B_xF\right)\right)\left(0\right) 
    = \left(B\left(B_0F\right)\right)\left(x\right) \\
    &= \left(BF\right)\left(x\right).
  \end{align*}
  This means that $F'$ is again in $V$. Because of $\dim V<\infty$
  the functions $F,F',\dots,F^{(\dim V)}$ are linearly dependent. Hence $F$ solves 
  a linear ODE with constant coefficients.

  $\emph{2.} \Rightarrow \emph{1.:}$
  The function $F$ solves a linear ODE with constant coefficients. 
  For every $t\in \R$ this ODE is solved by $F(\cdot-t)$ as well. Since the space of solutions 
  is finite-dimensional so is $\spn{F(\cdot-t)\;|\;t\in\R}$.

  $\emph{2.} \Leftrightarrow \emph{3.:}$
  This follows from standard linear ODE theory.
\end{proof}

%

\begin{lem} \label{specialspanlem}
  Let $F:\R\to\R$ be smooth, $2\pi$-periodic and even. Assume that the vector space 
  $\spn{F(\cdot-t)\;|\;t\in\R}$ is of finite dimension. Then there are 
  $k\in\N$, $Q_i\in\R$ and $\beta_i\in\N$ such that
  $$\forall\; x\in\R:\;\;F(x) = \sum_{i=1}^k Q_i\cos \beta_i x.$$
\end{lem}

\begin{proof}
  By Lemma \ref{spanlem} and the fact that $F$ is $2\pi$-periodic and even
  we get $k\in\N$, $Q_i\in\R$ and $\beta_i\in\R$ with the desired property.
  We only need to show that $\beta_i\in\N$. We may assume that the 
  $\beta_i$ are distinct and that $Q_i\neq 0$. Fix an $x\in\R$.
  Then $\cos\beta_1 x,\dots,\cos\beta_k x$ and $\sin\beta_1 x,\dots,\sin\beta_k x$ are linearly independent.
  Because of the $2\pi$-periodicity of $F$ we get
  \begin{align*}
    0 &= F(x-2\pi) - F(x+2\pi) \\
    &= \sum_{i=1}^k Q_i(\cos\beta_i x \cos 2\pi\beta_i + \sin\beta_i x \sin 2\pi\beta_i)
     -\sum_{i=1}^k Q_i(\cos\beta_i x \cos 2\pi\beta_i - \sin\beta_i x \sin 2\pi\beta_i)\\
    &= \sum_{i=1}^k 2Q_i\sin\beta_i x \sin 2\pi\beta_i.
  \end{align*}
  This yields $\sin 2\pi\beta_i=0$. Hence we get
  $$0 = F(x-2\pi) - F(x) = \sum_{i=1}^k Q_i\cos\beta_i x \cos 2\pi\beta_i -\sum_{i=1}^k Q_i\cos\beta_i x = \sum_{i=1}^k Q_i\cos\beta_i x (\cos 2\pi\beta_i-1).$$
  This yields $\cos 2\pi\beta_i=1$ and hence the claim.
\end{proof}

%

\begin{lem} \label{chebpoly}
  For every $m\in\N$ there are $a_{m,1},\dots,a_{m,m}\in\R$ such that
  $$\forall\; x\in\R:\;\; \cos mx = \sum_{k=1}^m a_{m,k}\cos^kx.$$
\end{lem}

\begin{proof}
  We can prove the claim by induction on $m$. For $m=1$ we have
  $a_{1,1}=1$. 
  If the claim is true for $1,\dots,m$ then 
  because of
  \begin{align*}
    \cos (m+1)x+\cos (m-1)x&= \cos mx \cos x -\sin mx\sin x + \cos mx\cos x +\sin mx \sin x \\
    &= 2\cos mx \cos x
  \end{align*}
  we have for all $x\in \R$
  \begin{align*}
    \cos (m+1)x = {}& -\cos (m-1)x + 2\cos mx \cos x \\ 
    = {}& -\sum_{k=1}^{m-1} a_{m-1,k}\cos^kx + 2\cos x\sum_{k=1}^m a_{m,k}\cos^kx \\ 
    = {}& -\sum_{k=1}^{m-1} a_{m-1,k}\cos^kx + 2\sum_{k=1}^m a_{m,k}\cos^{k+1}x \\ 
    = {}& -a_{m-1,1}\cos x+\sum_{k=2}^{m-1} (2a_{m,k-1}-a_{m-1,k})\cos^kx \\
    & + 2 a_{m,m-1}\cos^{m}x + 2 a_{m,m}\cos^{m+1}x.
  \end{align*}
\end{proof}

\begin{lem}[Gau{\ss}-Lucas' Theorem, {\cite[Theorem 2.1.1]{schmeisser_lucas}}] \label{lucasthm}
  If $P:\C\to\C$ is a non-constant polynomial with complex coefficients, 
  all roots of $P'$ belong to the convex hull of the set of roots of $P$.
\end{lem}

\begin{proof}
  Set $m:=\deg P$ and let $\zeta_1,\dots,\zeta_m\in\C$ be the (not necessarily distinct)
  roots of $P$.
  We can write 
  $$\forall\; z\in \C : \;\; P(z)= A\prod_{i=1}^m (z-\zeta_i)$$
  where $A\in \C$ is the leading coefficient of $P$.
  First fix a $w\in\C$ with $P'(w)=0$ and $P(w)\neq 0$. We have
  $$0=\frac{P^\prime(w)}{P(w)}= \sum_{i=1}^m \frac{1}{w-\zeta_i}
  =\sum_{i=1}^m \frac{\overline{w}-\overline{\zeta_i} } {\vert w-\zeta_i\vert^2}.$$
  This implies
  $$\left(\sum_{i=1}^m \frac{1}{\vert w-\zeta_i\vert^2}\right)\overline{w}
  =\sum_{i=1}^m \frac{1}{\vert w-\zeta_i\vert^2}\overline{\zeta_i}$$
  and by taking conjugates
  $$w=\frac{1}{\left(\sum_{i=1}^m \frac{1}{\vert w-\zeta_i\vert^2}\right)}\sum_{i=1}^m \frac{1}{\vert w-\zeta_i\vert^2}\zeta_i.$$
  Hence we get $w\in\conv{\zeta_1,\dots,\zeta_m}$. 
  Now assume that $P'(\zeta_j)=0$ for some $1\le j\le m$. 
  Since $\zeta_j\in\conv{\zeta_1,\dots,\zeta_m}$ we are done. 
\end{proof}

\newpage


\addcontentsline{toc}{section}{References}

\bibliographystyle{amsalpha}
\bibliography{lit.diplarb}

\end{document}